\documentclass[english,10pt]{amsart}
\usepackage[english]{babel}
\usepackage{tikz}
\usepackage[matrix,arrow]{xy}
\xyoption{all}
\usepackage{amscd,amssymb,amsfonts,amsmath}

\usepackage{graphics}
\usepackage{epsfig}
\usepackage{mathrsfs}

\newcommand\mypagesizel{
\textwidth= 6.5in
\textheight=9in
\voffset-.55in
\hoffset -0.75in
\marginparwidth=56pt
}

\newcommand{\p}[0]{{\mathbb P}}

\renewcommand{\phi}{\varphi}

\newcommand{\G}{\Gamma}

\newcommand{\cD}{\mathcal{D}}

\newcommand{\cH}{\mathcal{H}}

\newcommand{\cM}{\mathcal{M}}

\newcommand{\cX}{\mathcal{X}}

\newcommand{\sF}{\mathscr{F}}
\newcommand{\sG}{\mathscr{G}}
\newcommand{\sH}{\mathscr{H}}

\newcommand{\sJ}{\mathscr{J}}
\newcommand{\sK}{\mathscr{K}}

\newcommand{\sO}{\mathscr{O}}

\mypagesizel

\newtheorem{thm}{Theorem}[section]

\newtheorem{lemma}[thm]{Lemma}
\newtheorem{cor}[thm]{Corollary}

\newtheorem{prop}[thm]{Proposition}

\newtheorem*{thm*}{Theorem}

\theoremstyle{definition}

\newtheorem{defn-thm}[thm]{Definition-Theorem} 
\newtheorem{defn-lemma}[thm]{Definition-Lemma}

\theoremstyle{remark}
\newtheorem{rem}[thm]{Remark}

\newtheorem*{not-and-def}{Notation and definitions}

\numberwithin{equation}{section}

\begin{document}

\title[]{Lagrangian fibrations on symplectic fourfolds}

\author{Wenhao OU} 

\address{Wenhao OU: Institut Fourier, UMR 5582 du
  CNRS, Universit\'e Grenoble 1, BP 74, 38402 Saint Martin
  d'H\`eres, France} 

\email{wenhao.ou@ujf-grenoble.fr}

\subjclass[2010]{14J17, 14J26}

\begin{abstract}
We prove that there are at most two possibilities for the base of a Lagrangian fibration from a complex projective irreducible symplectic fourfold.
\end{abstract}

\maketitle

\tableofcontents

\section{Introduction}

A variety  $M$ (defined over a field $k$) of dimension $2n$ equipped with an everywhere non-degenerate closed two form $\Omega \in H^0(M, \Omega_M^2)$ is called a \textit{symplectic} variety. Note that in this case $\Omega^{\wedge n} \in H^0(M, \Omega_X^{2n})$ is non-zero everywhere. Hence $\Omega_M^{2n} \cong \sO_M$.  A subvariety $N \subseteq M$ of dimension $n$ is said to be \textit{Lagrangian} if the restriction of $\Omega$ on the smooth locus of $N$ is identically zero. A fibration $f:M\to X$ from a symplectic variety to a normal variety $X$ is called a \textit{Lagrangian fibration} if every component of every fiber of $f$ is a Lagrangian subvariety of $M$.

A simply connected complex projective manifold $M$ is called a \textit{projective irreducible symplectic manifold} if $M$ has a symplectic $2$-form $\Omega$, and $H^0(X, \Omega_X^2)=\mathbb{C}\Omega$. Any non trivial fibration from a  projective irreducible symplectic manifold $M$ to a normal variety $X$ is a Lagrangian fibration  by the following theorem.

\begin{thm}[{\cite[Thm. 2]{Mat99}}]
\label{thm-fibration-is-lagrangian}
Let $f:M\to X$ be a fibration from a complex projective irreducible symplectic manifold $M$ of dimension $2n$ to a normal variety $X$. Assume that $0<\mathrm{dim}\, X<2n$. Then 
\begin{enumerate}
\item[(1)] $X$ is a $\mathbb{Q}$-factorial klt Fano variety of dimension $n$ with Picard number  $1$.

\item[(2)] $f$ is a  Lagrangian fibration.  
\end{enumerate}
\end{thm}

Consider a Lagrangian fibration from a complex projective irreducible symplectic manifold $M$ to a normal variety $X$. We are interested in the base variety $X$.  In all examples of Lagrangian fibrations, $X$ is always isomorphic to $\p^n$. It is natural to ask if  this is true for all Lagrangian fibrations. On the one hand, if we assume that the base is smooth, then it is  proved by Hwang (see \cite[Thm. 1.2]{Hwa08}) that the base is always a projective space. On the other hand, if the irreducible symplectic manifold $M$ is of $K3^{[n]}$ or generalized Kummer deformation type, then the answer is also positive (see \cite[Cor. 1.1]{Mat13}, \cite{Mar13}, \cite{BM13} and \cite{Yos12}). However, it is still unclear if the base is always smooth. Moreover there are also irreducible symplectic manifolds which are not of $K3^{[n]}$ or generalized Kummer deformation type (see \cite{O99} and \cite{O03}). Thus the answer to this question is still unknown in the general setting.  In this paper, we study the case when $M$ has dimension $4$. We prove the following theorem.

\begin{thm}
\label{main-thm}
Let $f:M\to X$ be a Lagrangian fibration from a  projective irreducible symplectic manifold $M$ of dimension $4$ to a normal surface $X$. Then either $X \cong \p^2$ or $X \cong S^n(E_8)$.
\end{thm}

The surface $S^n(E_8)$ is the unique Fano surface with exactly one singular point which is Du Val of type $E_8$, and   two nodal rational curves in its anti-canonical  system. For more details on this surface, see the end of section \ref{surface-pre}. We will prove the theorem in several steps.  We will first prove the following theorem (see section \ref{Reflexivity of higher direct images}), which is a generalisation of a result of Koll\'ar (see \cite[Cor. 3.9]{Kol86b}).

\begin{thm}
\label{thm-reflexive-higher-direct-image}
Let $g:V \to W$ be a projective equidimensional fibration between complex normal quasi-projective varieties. Assume that $V$ is smooth, $\omega_V\cong \sO_V$, and that $W$ is Cohen-Macaulay and $\mathbb{Q}$-Gorenstein. Let $C$ be a $\mathbb{Q}$-Cartier  Weil divisor on $W$, and let $E=g^*C$. Then $R^ig_*\sO_V(E)$ is reflexive for all $i\geqslant 0$.
\end{thm}

In particular, if $f:M\to X$ is a Lagrangian fibration, then $R^if_*\sO_M$ is reflexive. Similarly to the proof of \cite[Thm. 1.3]{Mat05}, we can show that $R^if_*\sO_M\cong \Omega^{[i]}_X$. Afterwards, we will prove the following result (see section \ref{Cohomology properties}).

\begin{thm}
\label{thm-cohomology}
Let $f:M\to X$ be a Lagrangian fibration from a complex projective irreducible symplectic manifold $M$ to a normal projective variety $X$. Let $H$ be a  $\mathbb{Q}$-ample integral Weil divisor in $X$, and let $D=f^*H$. Then for all $j>0$ and $i\geqslant 0$, we have $$h^j(X, R^if_*(\sO_M(D)))=0  \ \mathrm{and}\  h^i(M,\sO_M(D)) = h^0(X,\Omega_X^{[i]} [\otimes]\sO_X(H)).$$
\end{thm}

As a corollary, we conclude that if $\mathrm{dim}\, M=4$, then  $$h^0(X,\sO_X(H)) -h^0(X, \Omega_X^{[1]} [\otimes] \sO_X(H))  +h^0(X,\sO_X(H+K_X)) = 3$$ for any $\mathbb{Q}$-ample Weil divisor $H$ on $X$. In the end, we will   prove the following classification result on rational surfaces and   deduce Theorem \ref{main-thm} (see section \ref{Rational surfaces}).

\begin{thm}
\label{thm-rat-surface}
Let $X$ be a complex Fano surface with klt singularities whose smooth locus is algebraically simply connected. Assume that $X$ has Picard number $1$, and   $$h^0(X,\sO_X(H)) -h^0(X, \Omega_X^{[1]} [\otimes] \sO_X(H))  +h^0(X,\sO_X(H+K_X)) = 3$$ for any $\mathbb{Q}$-ample integral Weil divisor $H$. Then either $X \cong \p^2$  or $X \cong S^n(E_8)$.
\end{thm}

Note that the converse of Theorem \ref{thm-rat-surface} also holds, see Remark \ref{rem-surface-E8}.

\section{Notation  of the paper}
In this paper, a variety is an integral scheme of finite type over a field. If $X$ and $T$ are two schemes, then $X$ is called a $T$-scheme if there is a morphism $X\to T$. Let $A$ be a ring. Then we may call $X$ a $A$-scheme if $X$ is a $(\mathrm{Spec}\, A)$-scheme. 

Let $X$, $Y$ be two $T$-schemes, let $\sF$ and $\sG$ be two quasi-coherent sheaves on $X$,   let $f:X \to Y$ be a morphism of $T$-schemes, let $\alpha: \sF \to \sG$ be a morphism of sheaves, let $t$ be any point of $T$, and let $\kappa$ be the residue field of $t$. There is a natural morphism $\iota: \mathrm{Spec}\,  \kappa \to T$. We denote the fiber product of $X$, $Y$, $\sF$, $\sG$, $f$ and $\alpha$ under the base change $\iota$ by $X_t$, $Y_t$, $\sF_t$, $\sG_t$, $f_t$ and $\alpha_t$. If $U$ is an open subset of $T$, similarly we denote the fiber product of $X$, $Y$, $\sF$, $\sG$, $f$ and $\alpha$ under the base change $U \hookrightarrow T$ by $X_U$, $Y_U$, $\sF_U$, $\sG_U$, $f_U$ and $\alpha_U$.

If $C^{\bullet}$ is a complex (of modules, of sheaves), then we denote the $i$-th cohomology of $C^{\bullet}$ by $H^i(C^{\bullet})$.   If $\sF$ is a coherent sheaf over some noetherian scheme $X$, then we denote by $H^i(X,\sF)$ the $i$-th cohomology group of the sheaf $\sF$. If $f:X\to Y$ is a morphism of noetherian schemes, then we denote the $i$-th higher direct image of $\sF$ by $R^if_*\sF$. Moreover, if $X$ is a variety defined over a field $k$, we denote by $h^i(X,\sF)$ the dimension of $H^i(X,\sF)$ over $k$. If $X$ is a complex algebraic variety, we denote  by $b_i(X)$ the $i$-th Betti number of $X$.

For a coherent sheaf $\mathscr{F}$ on a normal variety $X$ of dimension $n$, we denote by $\mathscr{F}^{*}$ the dual of $\mathscr{F}$. The sheaf $\sF$ is said to be reflexive if   $\sF \cong \sF^{**}$. In particular, $\sF^{**}$ is reflexive, and we  call it the reflexive hull of $\sF$. We denote $(\mathscr{F}^{\otimes m})^{**}$  by $\mathscr{F}^{[\otimes m]}$  for any $m>0$. For two coherent sheaves $\mathscr{F}$ and $\mathscr{G}$, let $\mathscr{F} [\otimes] \mathscr{G} = (\mathscr{F} \otimes \mathscr{G})^{**}$. We denote by $\Omega_X^1$ the sheaf of K\"ahler differentials of $X$. Let $\Omega_X^{[i]}=(\Omega_X^i)^{**}$ for all $i>0$. We denote by $K_X$ a canonical divisor of $X$. Write $\omega_X=\sO_X(K_X) \cong \Omega_X^{[n]}$. If $X$ is complex projective Cohen-Macaulay variety, then $\omega_X$ is   the dualizing sheaf.

A normal complex variety $X$ is said to be Gorenstein if all its local rings are Gorenstein rings. Note that a normal complex variety $X$ is Gorenstein if and only if $X$ is Cohen-Macaulay  and $K_X$ is Cartier. A normal complex variety $X$ is said to be $\mathbb{Q}$-Gorenstein if $K_X$ is a $\mathbb{Q}$-Cartier divisor.

The Picard number of  a normal   variety $X$ is the dimension of the $\mathbb{Q}$-vector space $\mathrm{Pic}(X)\otimes _{\mathbb{Z}} \mathbb{Q}$.  The variety is called a Fano variety if $-K_X$ is a $\mathbb{Q}$-ample divisor. If $X$ is a Fano variety, and $T$ is a curve in $X$, then $T$ is called a minimal curve if its intersection number with $-K_X$ is minimal. We denote by $\mathrm{Cl}(X)$ the Weil divisor class group of $X$.

A fibration $f:X\to Y$ between normal varieties is a surjective proper morphism whose fibers are connected. A Fano fibration is a projective fibration $f:X \to Y$ such that $-K_X$ is relatively ample. A Mori fibration is a Fano fibration whose relative Picard number is $1$. 

Let $f:X \to Y$ be an equidimensional fibration between normal varieties. Let $D$ be a Weil divisor on $Y$. If $Y_{0}$ is the smooth locus of $Y$, then $D|_{Y_{0}}$ is a Cartier divisor. Let $X_0=f^{-1}(Y_{0})$. Then $(f|_{X_0})^*(D|_{Y_{0}})$ is a Cartier divisor in $X_0$. Let $f^*D$ be the closure of $(f|_{X_0})^*(D|_{Y_{0}})$ in $X$. Then $f^*D$ is an integral Weil divisor in $X$. Assume that $D$ is $\mathbb{Q}$-Cartier and let $k$ be a positive integer   such that $kD$ is  Cartier. Then we may also define $f^*D$  by $f^*D=\frac{1}{k}f^*(kD)$. These two definitions of $f^*D$ coincide   under the assumption that $f$ is equidimensional.

\section{Cohomology properties}
\label{Cohomology properties}

The aim of this section is to prove  Theorem \ref{thm-cohomology}. One of the application of this theorem is the following corollary.

\begin{cor}
\label{cor-cohomology}
Let $f:M\to X$ be a Lagrangian fibration from a  projective irreducible symplectic manifold $M$ of dimension $2n$ to a normal projective variety $X$. Let $H$ be a  $\mathbb{Q}$-ample integral Weil divisor in $X$. Then  $\sum _{i=0}^{n} (-1)^{i}h^0(X,\Omega_X^{[i]} [\otimes]\sO_X(H)) = n+1$.
\end{cor}

\begin{proof}[{Proof of Corollary \ref{cor-cohomology}}]
Let $D=f^*H$. By Theorem \ref{thm-cohomology}, we have $h^i(M,\sO_M(D)) =h^0(X,\Omega_X^{[i]} [\otimes]\sO_X(H))$ for all  $i\geqslant 0$. Hence we only have to prove that $\sum _{i=0}^{n} (-1)^{i}h^i(M,\sO_M(D)) = n+1$. Note that Theorem \ref{thm-cohomology} also shows that $h^i(M,\sO_M(D))=0$ for $i >n$. Hence it is enough to prove that $$\sum _{i=0}^{2n} (-1)^{i}h^i(M,\sO_M(D)) = n+1.$$ However, $\sum _{i=0}^{2n} (-1)^{i}h^i(M,\sO_M(D))$ is the Euler characteristic $\chi(\sO_M(D))$ of the Cartier divisor $D$. Thus it suffices to prove that $\chi(\sO_M(D))=n+1$. 

Let $q_M$ be the Beauville-Bogomolov quadric form on $M$. Then $q_M(D)=0$ since the intersection number $D^{2n}$ is zero. Moreover, we have $\chi(\sO_M(D))=\sum \frac{a_i}{(2i)!}q_M(D)^i$, where the $a_i$'s are complex numbers depending only on $M$ (see \cite[1.11]{Huy99}). Hence $\chi(D) = \chi(\sO_M) =n+1$.
\end{proof}

\textit{Outline of the proof of Theorem \ref{thm-cohomology}.} We will first show that $R^if_*\sO_M(D) \cong \Omega_X^{[i]} [\otimes]\sO_X(H)$ for all $i\geqslant 0$ (see Proposition \ref{prop-reflexive-char-0}). Hence if $h^j(X, R^if_*\sO_M(D))=\{0\}$ for all $j>0$ and $i\geqslant 0$, then the Leray spectral sequence calculating $H^i(M,\sO_M(D))$ degenerates at $E_2$, and we obtain that $h^i(M,\sO_M(D)) = h^0(X,\Omega_X^{[i]} [\otimes]\sO_X(H))$ for all $i\geqslant 0$. Thus, in order to prove Theorem \ref{thm-cohomology}, it is enough to prove that $h^j(X, R^if_*\sO_M(D))=0$ for $j>0$ (Note that if $D$ is Cartier, then this is true by \cite[Thm. 2.1]{Kol86a} for $\omega_M\cong \sO_M$). 

If we assume that $\mathrm{dim}\, M=4$, then we can prove this vanishing by covering tricks and a theorem due to Koll\'ar (See Remark \ref{rem-vanishing-dim-2}). In higher dimension, we will prove  by reduction modulo $p$. We can reduce to the case when $f:M\to X$, $H$ and $D$ are defined over an algebraic number field $K$. In this case, there is a $\mathbb{Z}$-algebra $A\subseteq K$, which is a localisation of the ring of  integers of $K$, of Krull dimension $1$  such that the field of fraction of $A$ is $K$, and the coefficients of the equations defining $f:M\to X$, $H$ and $D$ are contained in $A$. Let $T=\mathrm{Spec}\, A$, and let $\eta$ be its generic point. Then $f$ induces a morphism $\phi: \cM \to \cX$ of $T$-schemes such that $\phi_{\eta}$ coincides with $f$ after a field extension. Let $\cD $ be the divisor on $\cM$  induced by $D$, and let $\cH $ be the divisor on $\cX$  induced by $H$. We will show that $h^j(\cX_t, R^i(\phi_t)_*\sO_{\cM_t}(\cD_t))=0$ for all $j>0$, $i\geqslant 0$ and general closed point $t\in T$. By a base change property (See Lemma \ref{lem-vanishing-from-general-to-generic}), this implies that $h^j(\cX_{\eta}, R^i(\phi_{\eta})_*\sO_{\cM_{\eta}}(\cD_{\eta}))=0$. Then we obtain $h^j(X, R^if_*\sO_M(D))=0$ since $f$ coincide with $\phi_{\eta}$ after a field extension, and a field extension is a flat base change.

We will divide this section into four parts. In the first part, we will develop some tools for the proof of Theorem \ref{thm-cohomology}. In the second part, we will show the reflexivity of some higher direct images. Note that for general closed $t\in T$, the varieties $\cX_t$ and $\cM_t$ are defined over a field of positive characteristic. We will prove that $h^j(\cX_t, R^if_*\sO_{\cM_t}(\cD_t))=0$ in the third part.  In the last subsection we will complete the proof of Theorem \ref{thm-cohomology}.

\subsection{Base change properties}  In this subsection, we will prove some properties on cohomology and base change. The aim is to prove the following proposition.

\begin{prop}
\label{prop-base-change}
Let $g:V\to W$ be a projective morphism of integral $T$-schemes of finite type, where $T$ is an integral  noetherian scheme. Let $\sF$ be a coherent sheaf on $V$. Then there is a non-empty open subset $T_0$ of $T$ such that  there is a natural isomorphism $(R^ig_*\sF)|_{W_t} \cong R^i(g_{t})_{*}\sF_t$ for any point $t\in T_0$ and any $i\geqslant 0$, where $g_t: V_t \to W_t$ is the restriction of $g$ over $t$, and $\sF_t$ is the restriction of $\sF$ on $V_t$.
\end{prop}

\centerline{
\xymatrix{
V_t  \ar[r] \ar[d]^{g_t}
&V   \ar[d]^g
\\
W_t  \ar[r]\ar[d]  
&W \ar[d]
\\ 
{\{}t{\}} \ar[r] 
&T
}
}

This proposition  is similar to \cite[Thm. III.12.8]{Har77}.  In fact, our proof follows that of \cite[Thm. III.12.8]{Har77}. We will need two lemmas.

\begin{lemma}
\label{lemma-loc-free-complex}
Let $A$ be a noetherian ring. Let $V$ be a projective $A$-scheme of finite type, and let $\sF$ be a coherent sheaf on $V$. Then there is a complex, bounded from above, of finitely generated free $A$-modules $L^{\bullet}$ such that $H^i(L^{\bullet}) \cong H^i(V, \sF)$ for all $i\geqslant 0$.
\end{lemma}

\begin{proof}
Let $(U_i)$ be an affine open cover of $V$. Let $C^{\bullet}$ be the \v Cech complex of $\sF$ with respect to $(U_i)$. Then $H^i(C^{\bullet}) \cong H^i(V, \sF)$ for all $i\geqslant 0$ (See \cite[Thm. III.4.5]{Har77}). Since $V$ is projective, $H^i(V, \sF)$ is a finitely generated $A$-module for all $i$ (See \cite[Thm. III.5.2]{Har77}). Hence, by \cite[Thm. III.12.3]{Har77}, there is a complex of finitely generated free $A$-modules $L^{\bullet}$ such that $H^i(L^{\bullet}) \cong H^i(C^{\bullet})$ for all $i\geqslant 0$. Hence $H^i(L^{\bullet}) \cong H^i(V, \sF)$ for all $i\geqslant 0$.
\end{proof}

\begin{lemma}
\label{lem-flat-complex-base-change}
Let $B$ be a noetherian ring, and let $A$ be a $B$-algebra of finite type. Let $C^{\bullet}$ and $L^{\bullet}$ be two complexes, bounded from above, of  $A$-modules such that  $C^j$ and $L^j$ are $B$-flat for all $j$. Assume that there is a quasi-isomorphism $L^{\bullet} \to C^{\bullet}$. Then for  any $B$-module $R$ and any $i \geqslant 0$, there is a natural isomorphism $H^i(R \otimes_B  L^{\bullet}) \cong H^i(R \otimes_B  C^{\bullet})$.
\end{lemma}

\begin{proof}
It is enough to prove the lemma for finitely generated $B$-modules since every $B$-module is the direct limit of finitely generated modules and both $\otimes$  and $H^i$ commute with direct limits.

We will prove the lemma by induction on $i$. For $i$ large enough, $C^i$ and $L^i$ are both $0$. Hence the  lemma is true for large enough $i$. Assume that $H^{i+1}(R \otimes_B  L^{\bullet}) \cong H^{i+1}(R \otimes_B  C^{\bullet})$ for any finitely generated $B$-module $R$,  we will show that $H^{i}(R \otimes_B  L^{\bullet}) \cong H^{i}(R \otimes_B C^{\bullet} )$ for any finitely generated $B$-module $R$. Fix a finitely generated $B$-module $R$, then there is a free $B$-module $E$ and a finitely generated $B$-module $K$ such that we have an exact sequence $0\to K \to E \to R \to 0$. Since $C^j$ and $L^j$ are $B$-flat for all $j$, we get a commutative diagram of complexes with exact rows

\centerline{
\xymatrix{
0 \ar[r] & K \otimes L^{\bullet} \ar[r]\ar[d] & E \otimes L^{\bullet} \ar[d] \ar[r]& R \otimes L^{\bullet} \ar[d] \ar[r] & 0\\
0 \ar[r]  & K \otimes C^{\bullet} \ar[r] & E \otimes C^{\bullet} \ar[r] & R \otimes C^{\bullet} \ar[r] & 0
}
}

By taking the cohomology, we obtain the following diagram with exact rows,

\centerline{
\xymatrix{
H^{i}(K \otimes L^{\bullet})  \ar[r] \ar[d]
& H^i(E \otimes L^{\bullet}) \ar[r] \ar[d]
& H^i(R \otimes L^{\bullet}) \ar[r] \ar[d]
& H^{i+1}(K \otimes L^{\bullet})  \ar[r] \ar[d]
& H^{i+1}(E \otimes L^{\bullet})  \ar[d]\\
H^{i}(K \otimes C^{\bullet})  \ar[r]  
&H^i(E \otimes C^{\bullet}) \ar[r]  
& H^i(R \otimes C^{\bullet}) \ar[r]
& H^{i+1}(K \otimes C^{\bullet}) \ar[r] 
& H^{i+1}(E \otimes C^{\bullet}) 
}
}

Since $E$ is free, and $L^{\bullet}$ and $C^{\bullet}$ are quasi-isomorphic, the morphism from $H^i(E \otimes L^{\bullet})$ to $H^i(E \otimes C^{\bullet})$ is an isomorphism. By induction hypothesis, the last two vertical morphisms in the last diagram are isomorphisms. Hence $H^i(R \otimes_B L^{\bullet} ) \to H^i(R \otimes_B  C^{\bullet})$ is surjective. Since $R$ is chosen arbitrarily,   this implies that the first vertical morphism  is also surjective. Hence $H^i(R \otimes_B L^{\bullet} ) \cong H^i(R \otimes_B C^{\bullet})$. This completes the induction and the proof of the lemma
\end{proof}

We will now prove Proposition \ref{prop-base-change}.

\begin{proof}[Proof of Proposition \ref{prop-base-change}]
Without lost of generality, we may assume that $T=\mathrm{Spec}\, B$ is  affine. First we assume that $W=\mathrm{Spec}\, A$ is affine. Let $(U_i)$ be a finite affine open cover of $V$. Let $C^{\bullet}$ be the \v Cech complex of $\sF$ with respect to $(U_i)$. Let $L^{\bullet}$ be the complex of finitely generated free $A$-modules given by Lemma \ref{lemma-loc-free-complex}. Then $H^i(L^{\bullet}) \cong H^i(V, \sF)$ for all $i\geqslant 0$. There is a non-empty open affine subset $T_0=\mathrm{Spec}\, B$ of $T$ such that for any point $t\in T_0$ and for any $i\geqslant 0$, we have $H^i(A_t \otimes L^{\bullet}) \cong H^i(L^{\bullet}) \otimes A_t$ (See \cite[Cor. 9.4.3]{Gro66}), where $A_t$ is the structural ring of the affine scheme $W_t$. 

Since $V\to W$ is a separated morphism, the intersection of any two affine open subsets of $V$ is still affine. Hence for any $j$, both $C^j$ and $L^j$ are  finitely generated modules over finitely generated $B$-algebras.  Moreover, both $C^{\bullet}$ and $L^{\bullet}$ are bounded. Hence by shrinking $T_0$, we may assume that both $C^j$ and $L^j$ are $B$-flat for any integer $j$ (See \cite[Thm. 6.9.1]{Gro65}). By Lemma \ref{lem-flat-complex-base-change}, this implies that $$H^i(A_t \otimes C^{\bullet}) \cong H^i(A_t \otimes L^{\bullet}) \cong H^i(L^{\bullet}) \otimes A_t \cong H^i(C^{\bullet}) \otimes A_t.$$ Note that $A_t\otimes C^{\bullet}$ is the \v Cech complex of $\sF_t$ with respect to $((U_i)_t)$, where $(U_i)_t$ is the fiber of $U_i$ over $t\in T$. Hence, we have $$H^i(V, \sF) \otimes A_t \cong H^i(C^{\bullet}) \otimes A_t \cong H^i(A_t \otimes C^{\bullet}) \cong  H^i(V_t, \sF_t).$$ Since $R^ig_*\sF$ (\textit{resp}. $R^i(g_t)_*\sF_t$) is the coherent sheaf associated to the module $H^i(V, \sF)$ (\textit{resp}. $H^i(V_t, \sF_t)$), we obtain that for general $t\in T$, $(R^ig_*\sF)|_{W_t} \cong R^i(g_{t})_{*}\sF_t$.

Now we treat the general case. We recover $W$ by finitely many affine open subsets $W_1,...,W_k$. Let $V_j= g^{-1}(W_j)$,  $\sF_j=\sF|_{V_j}$ and $g_j=g|_{V_j}$  for all $j$. Then for every $j=1,...,k$, there is a non-empty open subset $T_j$ of $T$ such that for any point $t\in T_j$, we have the base change property for $t$, $\sF_j$ and $g_j:V_j\to W_j$. Let $T_0$ be the open set $\bigcap_{j=1}^kW_j$. Then for any $t\in T_0$ and any $i\geqslant 0$, the natural morphism $(R^ig_*\sF)|_{W_t} \to R^i(g_{t})_{*}\sF_t$ is an isomorphism.
\end{proof}

As a corollary of the proposition, we obtain the following result.

\begin{lemma}
\label{lem-vanishing-from-general-to-generic}
Let $g:V\to W$ be a projective morphism of integral $T$-schemes of finite type, where $T$ is an integral  noetherian scheme. Let $\eta$ be the generic point of $T$. Let $\sF$ be a coherent sheaf on $V$. Let $i$, $j$ be two non-negative integers. If $h^j(W_t, R^i(g_{t})_{*}\sF_t)= 0$ for general closed point $t\in T$, then $h^j(W_{\eta}, R^i(g_{\eta})_{*}\sF_{\eta})= 0$.
\end{lemma}

\begin{proof}
By Proposition \ref{prop-base-change}, we know that for general point $t\in T$,  $R^i(g_{t})_{*}\sF_t \cong (R^ig_{*}\sF)|_{W_t}$. Hence $H^j(W_t, (R^ig_{*}\sF)|_{W_t})= \{0\}$ for general closed point $t\in T$. Since $R^ig_{*}\sF$ is a coherent sheaf on $W$ (See \cite[Thm. III.5.2]{Har77}), by shrinking $T$, we may assume that $R^ig_{*}\sF$ is flat over $T$ (See \cite[Thm. 6.9.1]{Gro65}). By \cite[Cor. III.12.9]{Har77}, we obtain $R^jh_*(R^ig_{*}\sF)=0$, where $h$ is the  morphism from $W$ to $T$. By \cite[Cor. III.12.9]{Har77} again, we have $h^j(W_{\eta}, (R^ig_{*}\sF)|_{W_{\eta}})=0$. From Proposition \ref{prop-base-change}, we conclude that $h^j(W_{\eta}, R^i(g_{\eta})_{*}\sF_{\eta})= 0$.
\end{proof}

\subsection{Reflexivity of higher direct images}
\label{Reflexivity of higher direct images}
Consider a Lagrangian fibration $f$ from a smooth complex projective symplectic variety $M$ to a projective variety $X$. Let $H$ be a Weil divisor on $X$, and let $D=p^*H$. We will show that the sheaf $R^if_*\sO_M(D)$ is reflexive and is isomorphic to $\Omega_X^{i} [\otimes] \sO_X(H)$ for all $i\geqslant 0$ (Proposition \ref{prop-reflexive-char-0}). 

Moreover, let $B$ be a $\mathbb{Z}$-subalgebra of $\mathbb{C}$. Assume that $T=\mathrm{Spec}\, B$ is integral, and let $\eta$ be its generic point. Assume that $f:M \to X$, $H$ and $D$ are defined over $B$. There is a morphism $\phi:\cM\to \cX$ of $T$-schemes given by the equations defining $f:M\to X$. There is a divisor $\cH$ on $\cX$ given by the equations defining $H$. Let $\cD=\phi^* \cH$. Then we will prove that for general closed point $t\in T$, the sheaf $R^i(\phi_{t})_{*}\sO_{\cM_t}(\cD_t)$ is reflexive and isomorphic to  $R^i(\phi_{t})_{*}\sO_{\cM_t} [\otimes] \sO_{\cX_t}(\cH_t)$ for all $i\geqslant 0$ (Corollary \ref{cor-reflexive-char-p}).

\begin{prop}
\label{prop-reflexive-char-0}
Let $f$ be a projective equidimensional Lagrangian fibration from a complex smooth quasi-projective symplectic variety $M$ to a complex quasi-projective variety $X$ with $\mathbb{Q}$-factorial klt singularities. Let $H$ be a Weil divisor on $X$, and let $D=p^*H$. Then $R^if_*\sO_M(D)$  is isomorphic to $\Omega_X^{i} [\otimes] \sO_X(H)$ for all $i\geqslant 0$. 
\end{prop}

We will first prove   Theorem \ref{thm-reflexive-higher-direct-image} which implies that $R^if_*\sO_M(D)$ is reflexive.

\begin{proof}[{Proof of Theorem \ref{thm-reflexive-higher-direct-image}}]
Set $m=$dim $V$ and $k=$dim $W$.

First we assume that $W$ is Gorenstein  and $C$ is Cartier. Then by the projection formula, we have $R^ig_*\sO_V(E) \cong R^ig_*\sO_V\otimes \sO_W(C)$ for all $i\geqslant 0$. Hence it is enough to prove that $R^ig_*\sO_V$ is reflexive for all $i$. Let $W\hookrightarrow \bar{W}$ be a projective compactification of $W$ such that $\bar{W}$ is normal.  Let $\bar{g}:\bar{V} \to \bar{W}$ be a compactification  of $g$ such that $\bar{V}$ is smooth. As in the proof of \cite[Cor. 3.9]{Kol86b}, the torsion-free part of $R^{i}\bar{g}_*\sO_{\bar{V}}$ is $\sH om(R^{m-k-i}\bar{g}_*\omega_{\bar{V}}, \ \omega_{\bar{W}})$. Since $\sO_V\cong \omega_V$, we have $(R^{i}\bar{g}_*\sO_{\bar{V}})|_W \cong (R^{i}\bar{g}_*\omega_{\bar{V}})|_W$. Since $R^{i}\bar{g}_*\omega_{\bar{V}}$ is torsion-free by \cite[Thm. 2.1]{Kol86a}, we obtain that $(R^{i}\bar{g}_*\sO_{\bar{V}})|_W$ is torsion free. Hence we have $$R^ig_*\sO_V\cong  (R^{i}\bar{g}_*\sO_{\bar{V}})|_W \cong \sH om(R^{m-k-i}\bar{g}_*\omega_{\bar{V}}, \ \omega_{\bar{W}})|_W = \sH om(R^{m-k-i}{g}_*\omega_{{V}}, \ \omega_{{W}}).$$ Since $W$ is Gorenstein, $\omega_W$ is an invertible sheaf. Hence $R^ig_*\sO_V \cong \sH om(R^{m-k-i}{g}_*\omega_{{V}}, \ \omega_{{W}})$ is reflexive.

Now we will treat the general case. Let $w$ be a point in $W$. Then the problem is local around $w$. Hence we may replace $W$ by any open neighbourhood of $w$. As in \cite[Def. 5.19]{KM98}, by shrinking $W$ if necessary, there is a finite cover $p:W'\to W$ which is \'etale in codimension $1$ such that both $K_{W'}=p^*K_W$  and $C'=p^*C$ are Cartier. Let $V'$ be the normalisation of $V\times_W W'$. Let $g':V' \to W'$ and $q:V'\to V$ be the natural projections. Write $E'=q^*E$.

\centerline{
\xymatrix{
V'\ar[d]^{g'} \ar[r]^{q}& V \ar[d]^{g}\\
W' \ar[r]^{p} & W
}}
Since $p$ is \'etale in codimension $1$, $V$ is smooth, and $g$ is equidimensional, the morphism $q$ is also \'etale in codimension $1$. By the Zariski purity theorem (See \cite[Prop. 2]{Zar58}), this implies that $q$ is \'etale. Hence $V'$ is smooth, and $\sO_{V'}\cong \omega_{V'}$. Hence $Rg'_*\sO_{V'}(E')$ is reflexive. Hence $p_*(Rg'_*\sO_{V'}(E'))$ is reflexive since $p$ is finite (See \cite[Cor. 1.7]{Har80}). On the one hand, since $q$ is finite, from the Leray spectral sequence, we know that $R^ig_{*}({q}_{*}\sO_{V'}(E'))$ is isomorphic to ${p}_{*}(R^ig'_{*}\sO_{V'}(E'))$, which is reflexive. On the other hand, since $\sO_{V}$ is a direct summand of $q_*\sO_{V'}$, the sheaf $\sO_{M}(E)$ is a direct summand of $q_*\sO_{V'}(E')$ by the projection formula. Hence $R^ig_{*}\sO_{V}(E)$, as a direct summand of the reflexive sheaf $R^ig_{*}({q}_{*}\sO_{V'}(E'))$,  is reflexive on $W$.  
\end{proof}

Now we will complete the proof of Proposition \ref{prop-reflexive-char-0}.

\begin{proof}[Proof of Proposition \ref{prop-reflexive-char-0}]
As in \cite[2.12]{Mat05}, there is an open subset $U$ in $X$ which has the following properties
\begin{enumerate}
\item[(1)] codim $X\backslash U \geqslant 2$;
\item[(2)]  $\bigwedge^i ((R^1f_*\sO_M)|_U) \cong (R^if_*\sO_M)|_U$ for all $i>0$;
\item[(3)] $(R^1f_*\sO_M)|_U \cong \Omega_U^1$.
\end{enumerate}

Since $R^if_{*}\sO_{M}$ is reflexive by Theorem \ref{thm-reflexive-higher-direct-image}, and codim $X\backslash U \geqslant 2$, we have $R^if_*\sO_M\cong \Omega_X^{[i]}$ for all $i>0$ by \cite[Prop. 1.6]{Har80}. From the projection formula, $R^if_{*}\sO_{M}(D)|_{X_{0}}$ is isomorphic to $(\Omega_X^{[i]} [\otimes] \sO_X(H))|_{X_{0}}$, where $X_{0}$ is the smooth locus of $X$. Since $X$ is smooth in codimension $1$, from \cite[Prop. 1.6]{Har80} we conclude that $R^if_{*}\sO_{M}(D) \cong \Omega_X^i [\otimes] \sO_X(H)$.
\end{proof}

Now let $T$ be an integral affine noetherian scheme with generic point $\eta$. Consider a morphism $g:V\to W$ of normal integral $T$-schemes of finite type. Let $\sF$ be a coherent sheaf on $V$. In the remaining of this subsection, we will prove the following proposition.

\begin{prop}
\label{prop-reflexive-char-p}
With the notation as above, assume that $R^i(g_{\eta})_{*}\sF_{\eta}$ is reflexive. Then there is a non-empty open subset $T_0$ of $T$ such that for any $t\in T_0$ and any $i\geqslant 0$, the sheaf $R^i(g_{t})_{*}\sF_t$ is reflexive.
\end{prop}

One of the application of Proposition \ref{prop-reflexive-char-p} is the following corollary. 

\begin{cor}
\label{cor-reflexive-char-p}
Let $f:M\to X$ be a Lagrangian fibration from a complex projective symplectic manifold $M$ to a normal variety $X$. Let $H$ be a Weil divisor on $X$, and let $D=f^*H$.  Let $A$ be a subalgebra of $\mathbb{C}$ such that $f$, $H$ and $D$ are defined over $A$. Let $T=\mathrm{Spec}\, A$, and let $\eta$ be its generic point. The equations defining $f:M\to X$, $H$ and $D$ also define a morphism $\phi:\cM \to \cX$ of $T$-schemes, and two divisors $\cH$ and  $\cD$ on $\cX$ and $\cM$. Then for any general $t\in T$ and any $i \geqslant 0$, the sheaf $R^i(\phi_{t})_{*}\sO_{\cM_t}(\cD_t)$ is reflexive, and 
$R^i(\phi_{t})_{*}\sO_{\cM_t}(\cD_t) \cong R^i(\phi_{t})_{*}\sO_{\cM_t} [\otimes] \sO_{\cX_t}(\cH_t)$.
\end{cor}

\begin{proof}[{Proof of Corollary \ref{cor-reflexive-char-p}}]
From Proposition \ref{prop-reflexive-char-0}, we know that $R^if_*\sO_M(D)$ is reflexive for all $i$. Note that $\phi_{\eta}$, $\cH_{\eta}$, and $\cD_{\eta}$ coincide with $f$, $H$ and $D$ after a field extension. Since a field extension is a faithfully flat base change, by \cite[Prop. III.9.3]{Har77} and Lemma \ref{lem-reflexive-faithfully-flat} below, $R^i(\phi_{\eta})_*\sO_{\cM_{\eta}}(\cD_{\eta})$ is also reflexive for all $i$. From Proposition \ref{prop-reflexive-char-p}, we obtain that for any general $t\in T$ and any $i\geqslant 0$, $R^i(\phi_{t})_{*}\sO_{\cM_t}(\cD_t)$ is reflexive. 

By the projection formula, if $(\cX_t)_{0}$ is the smooth locus of $\cX_t$, then $$(R^i(\phi_{t})_{*}\sO_{\cM_t}(\cD_t))|_{(\cX_t)_{0}} \cong (R^i(\phi_{t})_{*}\sO_{\cM_t} [\otimes] \sO_{\cX_t}(\cH_t))|_{(\cX_t)_{0}}.$$ Since $\cX_{\eta}$ is normal, $\cX_t$ is normal for general $t\in T$ (See \cite[Prop. 9.9.4]{Gro66}). Hence for general $t \in T$, $R^i(\phi_{t})_{*}\sO_{\cM_t}(\cD_t) \cong R^i(\phi_{t})_{*}\sO_{\cM_t} [\otimes] \sO_{\cX_t}(\cH_t)$ by \cite[Prop. 1.6]{Har80}.
\end{proof}

\begin{lemma}
\label{lem-reflexive-faithfully-flat}
Let $f:V\to W$ be a faithfully flat morphism of noetherian integral schemes. Let $\sF$ be a  coherent sheaf on $W$. Then $\sF$ is reflexive if and only if $f^*\sF$ is reflexive.
\end{lemma}

\begin{proof}
As in the proof of \cite[Prop 1.8]{Har80}, we have $f^*(\sF^{**}) \cong (f^*\sF)^{**}$ since $f$ is flat. Since $f$ is faithfully flat, the natural morphism $\sF \to \sF^{**}$ is an isomorphism if and only if $f^*(\sF) \to f^*(\sF^{**})$ is. Hence $\sF$ is reflexive if and only if $f^*\sF \to (f^*\sF)^{**}$ is an isomorphism, that is, if and only if  $f^*\sF$ is reflexive.
\end{proof}

In order to prove Proposition \ref{prop-reflexive-char-p}, we will need several lemmas. The following lemma gives a criterion of reflexivity for coherent sheaves.

\begin{lemma}[{\cite[Prop. 1.1]{Har80}}]
\label{lem-char-reflexive}
A coherent sheaf $\sF$ on an integral noetherian scheme $V$ is reflexive if and only if it fits in an exact sequence $0\to \sF \to \sJ \to \sK$, where $\sJ$ and $\sK$ are locally free.
\end{lemma}

\begin{proof}

If we have such an exact sequence, then by \cite[Prop. 1.1]{Har80}, $\sF$ is reflexive since $\sK$ is   torsion-free. 

Conversely, assume that $\sF$ is reflexive. There are two locally free sheaves $\sJ$, $\sK$ and an exact sequence  $0\to \sF^{**} \to \sJ \to \sK$  (See \cite[Prop. 1.1]{Har80}). Since $\sF$ is reflexive, we have $\sF \cong \sF^{**}$. Hence we obtain an exact sequence $0\to \sF \to \sJ \to \sK$.
\end{proof}

In the three lemmas below, we prove that reflexivity is a  ``constructible''  property (Lemma \ref{lem-reflexive-general-fiber}).

\begin{lemma}
\label{lem-extend-sheaves}
Let $T=\mathrm{Spec}\, B$ be an integral affine noetherian scheme. Let $V\to T$ be an integral $T$-scheme of finite type. Let $\eta$ be the generic point of $T$. Let $\alpha^{\eta}: \sF^{\eta} \to \sG^{\eta}$ be a morphism of coherent sheaves on $V_{\eta}$. Then there is an open neighbourhood $T_0$ of $\eta$ and a morphism $\alpha : \sF \to \sG$ of coherent sheaves on $V_{T_0}$ such that $\sF_{\eta} = \sF^{\eta}$, $\sG_{\eta} = \sG^{\eta}$, and $\alpha_{\eta}=\alpha^{\eta}$. Moreover, if $\sF^{\eta}$ (\textit{resp.} $\sG^{\eta}$) is locally free, then $\sF$ (\textit{resp.} $\sG$) can be chosen to be locally free.
\end{lemma}

The idea of the proof is as follows. There are only finitely many equations which define $\alpha^{\eta}$, $\sF^\eta$ and $\sG^{\eta}$. Then there is a  localisation $B_0$ of $B$ such that all of these equations are defined over $B_0$. These equations  define  $\alpha:\sF \to \sG$ over $V_{T_0}$, where $T_0=\mathrm{Spec}\, B_0$.

\begin{proof}[{Proof of Lemma \ref{lem-extend-sheaves}}]

Let $K$ be the fraction field of $B$. Note that $V$ is a noetherian scheme. Hence $V$ can be covered by finitely many open affine subsets $U_1,...,U_k$. Then $V_{\eta}$ is covered by open affine subsets $(U_1)_{\eta},...,(U_k)_{\eta}$. We can recover $U_i\cap U_j$ by finitely many open affine subsets $R^{ij}_s=\mathrm{Spec}\, C^{ij}_{s}$. Let $A_i$ be the structural ring of $U_i$, and let $(A_i)_{\eta}$ be the structural ring of $(U_i)_{\eta}$.  Then for all $i$, $A_i$ is a $B$-algebra of finite type. Since $\sF^{\eta}$ is a coherent sheaf, over each $(U_i)_{\eta}$, there is an $(A_i)_{\eta}$-module $M_i^{\eta}$ of finite presentation such that $\sF^{\eta}|_{(U_i)_{\eta}}$ is the coherent sheaf associated to $M_i^{\eta}$. For every $i$, we have an exact sequence  $$ ((A_i)_{\eta})^{\oplus t_i} \to ((A_i)_{\eta})^{\oplus r_i} \to M_i^{\eta} \to 0 $$ where $r_i$ and $t_i$ are some non negative integers. The first morphism $\theta_i$ in the   sequence is defined by finitely many equations. Since $\sF^{\eta}$ is a coherent sheaf, there is an isomorphism $(\mu^{ij}_s)^{\eta}$ from $M^{\eta}_i \otimes _B C^{ij}_s$ to $M^{\eta}_j \otimes_B C^{ij}_s$. Then $(\mu^{ij}_s)^{\eta}$ is given by finitely many equations. There is an open affine subset $T_0=\mathrm{Spec}\, B_0$ of $T$ such that all of the equations defining the $\theta_i$'s and the $(\mu^{ij}_s)^{\eta}$'s are defined over $B_0$. Hence, for every $i$, we have an exact sequence of $(A_i \otimes _B B_0)$-modules, induced by the previous one,  $$ (A_i \otimes _B B_0)^{\oplus t_i} \to (A_i \otimes _B B_0)^{\oplus r_i} \to M_i \to 0.$$ The $(A_i \otimes _B B_0)$-module $M_i$ induces a coherent sheaf on $(U_i)_{T_0} = U_i\times_T T_0$. Moreover, since all of the equations defining the  $(\mu^{ij}_s)^{\eta}$'s are defined over $B_0$, the function $(\mu^{ij}_s)^{\eta}$ induces a transition function $\mu^{ij}_s$ from  $M_i \otimes _B C^{ij}_s$ to $M_j \otimes_B C^{ij}_s$. The modules $M_i$ and the isomorphisms $\mu^{ij}_s$ define a coherent sheaf  $\sF$ on $V_{T_0}$ such that $\sF_{\eta}$ is isomorphic to $\sF^{\eta}$. If $\sF^{\eta}$ is locally free, then $\sF$ can also be constructed to be locally free (for example, we can take $t_i=0$ for all $i$).

For the same reason as above, by shrinking $T_0$, we can construct a coherent sheaf  $\sG$ on $V_{T_0}$ such that $\sG_{\eta}$ is isomorphic to $\sG^{\eta}$. Assume that $\sG^\eta|_{(U_i)_{\eta}}$ is the sheaf associated to a $(A_i)_{\eta}$-module $N_i^{\eta}$ of finite presentation. Then there is a morphism of $(A_i)_{\eta}$-modules $e_i^{\eta}:M^{\eta}_i \to N^{\eta}_i$ such that $\alpha|_{(U_i)_\eta}$ is associated to $e_i^{\eta}$. Since $e_i^{\eta}$ is defined by finitely many equations, by shrinking $T_0$, we may assume that all of the   equations defining the $e_i^{\eta}$'s are defined over $B_0$. From these equations, we obtain morphisms $e_i:M_i \to N_i$. These morphisms induce a morphism $\alpha:\sF \to \sG$ such that $\alpha_{\eta}$ is the same as $\alpha^{\eta}$.
\end{proof}

With the same method, we can obtain the following lemma. The proof is left to the reader.

\begin{lemma}
\label{lem-extend-morphism-generic-global}
Let $T=\mathrm{Spec}\, B$ be an integral affine noetherian scheme. Let $V\to T$ be a $T$-scheme of finite type. Let $\eta$ be the generic point of $T$. Let $\sF$, $\sG$ be coherent sheaves on $V$. Assume that there is a morphism $\alpha^{\eta}: \sF_{\eta} \to \sG_{\eta}$ be a morphism of coherent sheaves on $V_{\eta}$. Then there is an open neighbourhood $T_0$ of $\eta$ and a morphism $\alpha : \sF_{T_0} \to \sG_{T_0}$ of coherent sheaves on $V_{T_0}$ which extends $\alpha^{\eta}$.
\end{lemma}

\begin{lemma}
\label{lem-reflexive-general-fiber}
Let $V$ be a $T$-scheme of finite type, where $T$ is an integral noetherian scheme. Let $\eta$ be the generic point of $T$. Let $\sF$ be a coherent sheaf on $V$ such that $\sF_{\eta}$ is a reflexive sheaves on $V_{\eta}$. Then   for general $t\in T$, $\sF_t$ is reflexive on $V_t$.
\end{lemma}

\begin{proof}
We may assume that $T$ is affine. Since $\sF_{\eta}$ is reflexive, by Lemma \ref{lem-char-reflexive} there are two locally free sheaves $\sJ^{\eta}$, $\sK^{\eta}$ such that we have an exact sequence $0\to \sF_{\eta} \stackrel{\beta^{\eta}}{\longrightarrow} \sJ^{\eta} \stackrel{\alpha^{\eta}}{\longrightarrow} \sK^{\eta}$. By Lemma \ref{lem-extend-sheaves}, there is an open neighbourhood $T_1$ of $\eta$ and a morphism $\alpha: \sJ \to \sK$ of locally free sheaves on $V_{T_1}$, such that  $\sJ_{\eta} = \sJ^{\eta}$, $\sK_{\eta} = \sK^{\eta}$, and $\alpha_{\eta}=\alpha^{\eta}$. From Lemma \ref{lem-extend-morphism-generic-global}, by shrinking $T_1$, we may assume that there is a morphism $\beta: \sF_{T_1} \to \sJ$ such that $\beta_{\eta}=\beta^{\eta}$.

Consider the morphisms  $\beta: \sF_{T_1} \to \sJ$ and $\alpha: \sJ \to \sK$. By \cite[Prop. 9.4.4]{Gro66}, the set $T_0$ of points $t\in T_1$ such that the sequence $0\to \sF_t\to \sJ_t\to \sK_t$ is exact is constructible. Moreover, we know that $\eta\in T_0$. Hence $T_0$ contains an open neighbourhood of $\eta$. For any $t\in T_0$, $\sJ_t$ and $\sK_t$ are locally free. Hence for any $t\in T_0$, the sheaf $\sF_t$ is reflexive on $V_t$ by Lemma \ref{lem-char-reflexive}. This completes the proof of the lemma.
\end{proof}

Combining Lemma \ref{lem-reflexive-general-fiber} and Proposition \ref{prop-base-change}, we can prove Proposition \ref{prop-reflexive-char-p}.

\begin{proof}[Proof of Proposition \ref{prop-reflexive-char-p}]
From Proposition \ref{prop-base-change}, there is an open subset $T_1$ of $T$ such that for any $t\in T_1$ and any $i\geqslant 0$, we have $R^i(g_{t})_{*}\sF_{t} \cong (R^ig_*\sF)|_{W_t}$. By hypothesis,  $R^i(g_{\eta})_{*}\sF_{\eta}$ is reflexive.  By Proposition \ref{prop-base-change}, this implies that $(R^ig_*\sF)|_{W_{\eta}}$ is reflexive. Hence by Lemma \ref{lem-reflexive-general-fiber},  there is an open subset $T_0$ of $T_1$ such that for all $t\in T_0$, the sheaf $(R^ig_*\sF)|_{W_t}$ is reflexive. Hence for all $t\in T_0$, $R^i(g_{t})_{*}\sF_t$ is reflexive.
\end{proof}

\subsection{Vanishing in positive characteristic}

We will prove a vanishing lemma (Lemma \ref{lem-vanishing-by-frobenius}) in positive characteristic. If $Y$ is a variety defined over a field of positive characteristic, then $Y$ is called Frobenius-split if $\sO_Y$ is a direct summand of $F_{abs*}\sO_Y$, where $F_{abs}:Y \to Y$ is the absolute Frobenius morphism (See \cite[\S 1.1]{BK07} for more details).

\begin{lemma}
\label{lem-vanishing-ample}
Let $Y$ be a normal projective variety. Let $\sF$ be a reflexive sheaf on $Y$ such that $\sF^{[\otimes m]}$ is a very ample invertible sheaf for some $m>0$. Let $\G$ be a coherent sheaf on $Y$. Then there is an integer $l$ such that $h^i(Y,\sG [\otimes] \sF^{[\otimes t]})=0$ for all $t\geqslant l$ and $i>0$.
\end{lemma}

\begin{proof}
Since $\sF^{[\otimes m]}$ is invertible, $\sF^{[\otimes ms]} \cong (\sF^{[\otimes m]})^{\otimes s}$ for all integer $s$ (See \cite[Prop. 1.6]{Har80}).  For any integer $r$ between $0$ and $m-1$, there is an integer $l_r>0$ such that  $H^i(Y,(\sG [\otimes] \sF^{[\otimes r]}) \otimes \sF^{[\otimes ms]})=\{0\}$ for all $i>0$ and $s\geqslant l_r$ (See \cite[Thm. III.5.2]{Har77}). Let $l=m(l_0+\cdots + l_{m-1})$, then for any integer $t\geqslant l$, we have $$\sG [\otimes] \sF^{[\otimes t]} = (\sG [\otimes] \sF^{[\otimes (t-m[\frac{t}{m}])]}) \otimes \sF^{[\otimes (m[\frac{t}{m}])]}$$ with $0\leqslant t-m[\frac{t}{m}] \leqslant m-1$, and $[\frac{t}{m}]\geqslant l_0+\cdots +l_{m-1}$. Hence $h^i(Y,\sG [\otimes] \sF^{[\otimes t]})=0$ for all $t\geqslant l$ and $i>0$.
\end{proof}

\begin{lemma}
\label{lem-vanishing-by-frobenius}
Let $g:Y\to Z$ be a morphism of normal projective varieties over a field of characteristic $p>0$. Let $L$ be a Weil divisor on $Z$ such that $mL$ is an ample Cartier divisor for some $m>0$. Assume that $Y$ is Frobenius-split. Then $h^j(Z,R^ig_*\sO_Y[\otimes] \sO_Z(L))=0$ for all  $j>0$ and $i\geqslant 0$.
\end{lemma}

\begin{proof}
Fix $j>0$ and $i \geqslant 0$. Let $F_Y^e:Y\to Y$ and $F_Z^e:Z\to Z$ be the compositions of $e$ absolute Frobenius morphisms for $e>0$. Then $\sO_Y$ is a direct summand of $(F_Y^e)_{*}\sO_Y$ since $Y$ is Frobenius-split. Hence we only have to prove $h^j(Z,R^ig_*((F_Y^e)_{*}\sO_Y)[\otimes] \sO_Z(L)) = 0$ for some $e>0$. Since $F_Y^e$ and $F_Z^e$ are finite morphisms, from the Leray spectral sequence, we have $$R^ig_*((F_Y^e)_{*}\sO_Y) \cong (F_Z^e)_{*}(R^ig_*\sO_Y).$$ Moreover, if $Z_{0}$ is the smooth locus of $Z$, then $L|_{Z_{0}}$ is a Cartier divisor, and $((F_Z^e)^*L)|_{Z_{0}} = (p^eL)|_{Z_{0}}$. Hence by the projection formula, $$((F_Z^e)_{*}(R^ig_*\sO_Y \otimes \sO_Z(p^eL)))|_{Z_{0}} \cong (((F_Z^e)_{*}(R^if_*\sO_Y)) \otimes \sO_Z(L))|_{Z_{0}}.$$  Since $(F_Z^e)_{*}(R^ig_*\sO_Y [\otimes] \sO_Z(p^eL))$ is a reflexive sheaf on $Z$ (See \cite[Cor. 1.7]{Har80}), and $Z$ is smooth in codimension $1$, by \cite[Prop. 1.6]{Har80}, we obtain $$(F_Z^e)_{*}(R^ig_*\sO_Y [\otimes] \sO_Z(p^eL))\cong ((F_Z^e)_{*}(R^ig_*\sO_Y)) [\otimes] \sO_Z(L).$$  Since $H^j(Z,R^ig_*((F_Y^e)_{*}\sO_Y)[\otimes] \sO_Z(L)) \cong
H^j(Z, ((F_Z^e)_{*}(R^ig_*\sO_Y)) [\otimes] \sO_Z(L))$, we have $$ H^j(Z,R^ig_*((F_Y^e)_{*}\sO_Y)[\otimes] \sO_Z(L))\cong 
 H^j(Z, (F_Z^e)_{*}(R^ig_*\sO_Y [\otimes] \sO_Z(p^eL)))$$ Since $F_Z^e$ is a finite morphism, the Leray spectral sequence shows that the right-hand side  is isomorphic to $H^j(Z, R^ig_*\sO_Y [\otimes] \sO_Y(p^eL)). $ Thus it is enough to prove that $h^j(Z,R^ig_*\sO_Y [\otimes] \sO_Z(p^eL))=0$ for some $e>0$. This is true by Lemma \ref{lem-vanishing-ample}.
\end{proof}

\subsection{Proof of Theorem \ref{thm-cohomology}} 
By the vanishing lemma in positive characteristic, we can prove Theorem \ref{thm-cohomology} in the special case when everything is defined over an algebraic number field.

\begin{lemma}
\label{lem-cohomology-number-field}
Let $f:M\to X$, $H$ and $D$ be as in Theorem \ref{thm-cohomology}. Assume further that everything is defined over an algebraic number field $K$. Then  $h^j(X, R^if_*\sO_M(D))=0$ for all $j>0$ and $i\geqslant 0$.
\end{lemma}

\begin{proof}
Since $X$ has $\mathbb{Q}$-factorial klt singularities (See Theorem \ref{thm-fibration-is-lagrangian}), by Proposition \ref{prop-reflexive-char-0}, $R^if_*\sO_M(D)$ is reflexive on $X$ for all $i\geqslant 0$.

There is a ring $A \subseteq K$ of Krull dimension $1$ such that the field of fraction of $A$ is $K$, and the coefficients of equations defining $f:M\to X$, $H$ and $D$ are contained in $A$. Let $T=\mathrm{Spec}\, A$, and let $\eta$ be its generic point. Let $\phi: \cM \to \cX$ be the morphism of $T$-schemes given by the equations defining $f:M\to X$.  Fix some $i\geqslant 0$ and some $j>0$.

Let $\cD $ and $\cH$ be the divisors on $\cM$ and $\cX$ given by the equations defining $D$ and $H$. Then for a general closed point $t\in T$, $\cM_t$ is a smooth symplectic variety defined over the residue field of $t$, which is of positive characteristic. Moreover, since $\cX_{\eta}$ is normal, by \cite[Prop. 9.9.4]{Gro66}, $\cX_t$ is normal for general $t$. From \cite[Thm. 1.1 and Lem. 3.3]{Dru12}, we know that $\cM_t$ is Frobenius-split for general closed point $t\in T$.  Hence, by Lemma \ref{lem-vanishing-by-frobenius}, we have $h^j(\cX_t, R^i(\phi_{t})_{*}\sO_{\cM_t}[\otimes] \sO_{\cX_t}(\cH_t))=0$ for general closed point $t\in T$. By Corollary \ref{cor-reflexive-char-p}, this shows that $h^j(\cX_t, R^i(\phi_{t})_{*}\sO_{\cM_t}(\cD_t))= 0$. 

From Lemma \ref{lem-vanishing-from-general-to-generic}, we obtain $h^j(\cX_{\eta}, R^i(\phi_{\eta})_{*}\sO_{\cM_{\eta}}(\cD_{\eta}))= 0$. Note that $\phi_{\eta}: \cM_{\eta} \to \cX_{\eta}$ coincides with $f:M \to X$ after a field extension. Since a field extension is a flat base change, we conclude that $h^j(X, R^if_{*}\sO_{M}(D))= 0$ (See \cite [Prop.III.9.3]{Har77}). We also obtain that $h^j(X, \Omega_X^{[i]} [\otimes] \sO_X(H))= 0$ by Proposition \ref{prop-reflexive-char-0}.
\end{proof}

Now we will complete the proof of Theorem \ref{thm-cohomology}.

\begin{proof}[Proof of Theorem \ref{thm-cohomology}]
Fix some $j> 0$ and some $i\geqslant 0$. Let $\bar{\mathbb{Q}}$ be the field of algebraic numbers over $\mathbb{\mathbb{Q}}$. Then there is a $\bar{\mathbb{Q}}$-algebra $B$ of finite type which contains all the coefficients of the equations defining $f:M\to X$, $D$ and $H$. Let $T=\mathrm{Spec}\, B$ and let $\eta$ be its generic point. Then the  equations induce a morphism $\phi:\cM\to \cX$ of $T$-schemes, a divisor $\cH$ on $\cX$ and a divisor $\cD$ on $\cM$.  Since $R^i\phi_*\sO_{\cM}(\cD)$ is a coherent sheaf on $\cX$, by shrinking $T$ we may assume that this sheaf is flat over $T$ (See \cite[Thm. 6.9.1]{Gro65}).

For general closed point $t\in T$, $\phi_t$ is a Lagrangian fibration from a smooth variety $\cM_t$ to a normal variety $\cX_t$. In addition, for general closed point $t$, we know that $\phi_t$, $\cH_t$ and $\cD_t$ are all defined over an algebraic number field since the residue field of $t$ is just $\bar{\mathbb{Q}}$. Hence by Lemma \ref{lem-cohomology-number-field}, we have $h^j(\cX_t, R^i(\phi_{t})_{*}\sO_{\cM_t}(\cD_t))=0$. By Lemma \ref{lem-vanishing-from-general-to-generic}, this shows that $h^j(\cX_{\eta}, R^i(\phi_{\eta})_{*}\sO_{\cM_{\eta}}(\cD_{\eta}))=0$. Note that $\phi_{\eta}:\cM_{\eta} \to \cX_{\eta}$ coincides with $f:M\to X$ after a field extension.  Since a field extension is a flat base change, we have $H^j(X, R^if_{*}\sO_{M}(D))=\{0\}$  (See \cite [Prop.III.9.3]{Har77}).

By Proposition \ref{prop-reflexive-char-0} and the Leray spectral sequence, this proves the theorem.
\end{proof}

\begin{rem}
\label{rem-vanishing-dim-2}
If we assume that $\mathrm{dim}\, M=4$ and $\mathrm{dim}\, X=2$, then we can prove Theorem \ref{thm-cohomology} without using reduction modulo $p$. In fact, we only have to prove that $h^j(X,R^if_*\sO_M(D))=0$ for all $j>0$ and $i\geqslant 0$. Since $\omega_M\cong \sO_M$, this is true by the following lemma.
\end{rem}

\begin{lemma}
\label{lem-vanishing-dim-2}
Let $g:V\to W$ be a  projective equidimensional  fibration from a smooth projective variety to a normal projective surface. Let $C$ be a $\mathbb{Q}$-ample Weil divisor on $W$, and let $E=g^*C$. Then $h^j(X,R^ig_*(\sO_V(E)\otimes \omega_V))=0$ for all $j>0$ and $i\geqslant 0$.
\end{lemma}

\begin{proof}
Let $k$ be a positive integer such that $kC$ is very ample. Then there is a prime Cartier divisor $B$ on $W$ such that 
\begin{enumerate}
\item[(1)] $\sO_W(B)\cong \sO_W(kC)$
\item[(2)] $B$ is smooth and is contained in the smooth locus of $W$
\item[(3)] the divisor $A=f^*B$ is smooth.
\end{enumerate}

Let $p:W'\to W$ be the cyclic cover with respect to   $\sO_W(B)\cong (\sO_W(C))^{\otimes k}$, and let $q:V'\to V$ be the cyclic cover with respect to   $\sO_V(A)\cong (\sO_V(E))^{\otimes k}$ (See \cite[Def. 2.52]{KM98}). Then $V'$ is smooth, $W'$ is smooth along $p^{-1}(B)$, and there is a projective equidimensional  fibration $g':V'\to W'$ induced by $g$. We obtain the following commutative diagram.

\centerline{
\xymatrix{
V'\ar[d]^{g'} \ar[r]^{q}& V \ar[d]^{g}\\
W' \ar[r]^{p} & W
}}

Let $E'=q^*E$ and $C'=p^*C$. Then $C'$ is linearly equivalent to the support of $p^*B$ which is a Cartier divisor. Since $C'$ is also ample, by  the projection formula and \cite[Thm. 2.1]{Kol86a}, we have $$h^j(W', R^ig'_*(\sO_{V'}(E')\otimes \omega_{V'}))=h^j(W', \sO_{W'}(C')\otimes R^ig'_*\omega_{V'})=0$$ for all $j>0$ and $i\geqslant 0$.

Since $q$ and $p$ are finite morphisms, by the Leray spectral sequence, we have $$h^j(W, R^ig_*(q_*(\sO_{V'}(E')\otimes \omega_{V'})))=0$$ for all $j>0$ and $i\geqslant 0$. Note that $q_*\omega_{V'}=\omega_V \otimes (\sum_{r=0}^{k-1}\sO_V(rE))$.  By the projection formula, we obtain that  $\sO_V(E)\otimes \omega_{V}$ is a direct summand of $q_*(\sO_{V'}(E')\otimes \omega_{V'})$. Hence,  $h^j(W, R^ig_*(\sO_{V}(E)\otimes \omega_{V}))=0$ for all $j>0$ and $i\geqslant 0$.
\end{proof}

\section{Rational surfaces}
\label{Rational surfaces}

In this section, we  work over $\mathbb{C}$, the field of complex numbers. A point in a variety will always refer  to a closed point. The aim is to prove Theorem \ref{thm-rat-surface}. Let $X$ be a surface satisfying the conditions in Theorem \ref{thm-rat-surface} and let $T$ be a  minimal curve in $X$. Then we have the  following lemma.

\begin{lemma}
\label{Weil-divisor-class}
Let $X$ be a Fano surface with klt singularities whose smooth locus is algebraically simply connected. Assume that $X$ has Picard number $1$, and that $h^0(X,\sO_X(H))  +h^0(X,\sO_X(H+K_X)) \geqslant 3$ for all $\mathbb{Q}$-ample Weil divisor $H$. Let $T$ be a minimal curve in $X$. Then the Weil divisor class group $\mathrm{Cl}(X)$ is generated by the class of $T$.
\end{lemma}

\begin{proof}
We have $h^0(X,\sO_X(-K_X))+h^0(X,\sO_X(K_X-K_X)) \geqslant 3$. Hence $h^0(X,\sO_X(-K_X)) \geqslant 2$. Let $D$ be a Weil divisor in $X$, and let $a$ be the largest integer such that $(D-aT)\cdot T > 0$. Set $G= D-aT$. Then $G$ is a  $\mathbb{Q}$-ample Weil divisor since $X$ has Picard number $1$. By hypothesis, we have  $$h^0(X,\sO_X(G))+h^0(X,\sO_X(K_X+G)) \geqslant 3.$$ Since $h^0(X,\sO_X(-K_X)) \geqslant 1$, we obtain that $h^0(X,\sO_X(G)) \geqslant h^0(X,\sO_X(K_X+G))$. Hence, $h^0(X,\sO_X(G)) \geqslant 1$. However, by the choice of $a$, we have $0< G \cdot T \leqslant T \cdot T$. Since $T$ is a minimal curve, this implies that $G$ is numerically equivalent to $T$. Hence $D$ is numerically equivalent to $(a+1)T$. By \cite[Lem. 2.6]{AD12}, there is a smallest positive integer $k$ such that $k(D-(a+1)T)$ is linearly equivalent to the zero divisor. 

There is a cyclic cover $X'\to X$ with respect to  $\sO_X(D-(a+1)T)^{\otimes k} \cong \sO_X$, which is \'etale in codimension $1$ (See \cite[Def. 2.52]{KM98}). Since the smooth locus of $X$ is algebraically simply connected,   we can only have $k=1$. Thus, we obtain that $D$ is linearly equivalent to $(a+1)T$. This shows that $\mathrm{Cl}(X)=\mathbb{Z}\cdot [T]$
\end{proof}

\textit{Outline of the proof of Theorem \ref{thm-rat-surface}.} Let  $T$ be a minimal curve.  Since $-K_X$ is ample, we have $h^0(X,\sO_X(T+K_X)) \leqslant 1$, and the  equality holds if and only if $-K_X$ is linearly equivalent to $T$. Hence, by Lemma \ref{Weil-divisor-class}, the conditions in Theorem \ref{thm-rat-surface} imply that 
\begin{enumerate}
\item either $h^0(X,\sO_X(-K_X))\geqslant 2$ and $\mathrm{Cl}(X)=\mathbb{Z}\cdot [K_X]$, 

\item or $h^0(X,\sO_X(T))\geqslant 3$, $h^0(X, \sO_X(T+K_X))=0$, and $\mathrm{Cl}(X)=\mathbb{Z}\cdot [T]$
\end{enumerate}
If we are in the first case, then we will   prove that $X$ has canonical singularities. After that, with the help of the classification of Fano surfaces with canonical singularities, we can show that   $X \cong S^c(E_8)$ or $X \cong S^n(E_8)$ (Theorem \ref{thm-surface-1}). In the second case, we will show directly that $X\cong \p^2$ (Theorem \ref{thm-surface-2}). After these two theorems, it remains to exclude the case of $S^c(E_8)$. The complete  proof of Theorem \ref{thm-rat-surface} will be given in section \ref{proof-surface}.

\subsection{Preliminaries}
\label{surface-pre}

In this subsection, we recall some basic properties of complex surfaces.  Let $X$ be a normal surface. A morphism $r: \widetilde{X} \to X$ is said to be a minimal resolution of singularities (or minimal resolution) if $\widetilde{X}$ is smooth, and $K_{\widetilde{X}}$ is relatively nef. For any complex surface, there is a unique minimal resolution, and every resolution of singularities factors through the minimal resolution. If $X$ is $\mathbb{Q}$-factorial, and $r:\widetilde{X} \to X$ is the minimal resolution, then we can write $K_{\widetilde{X}}=r^*K_X+\sum a_iE_i$, where the $E_i$'s are the $r$-exceptional divisors. For every $i$, $a_i$ is called the discrepancy of $E_i$ over $X$. An exceptional divisor $E_i$ is said to be crepant over $X$ if $a_i=0$. If $a_i=0$ (\textit{resp.} $-1<a_i<0$) for all $i$, then we say that $X$ has canonical singularities (\textit{resp.} klt singularities). Klt surface singularities are quotient singularities. In particular, they are $\mathbb{Q}$-factorial rational singularities.  Canonical surface singularities are Du Val singularities (or ADE singularities).  See \cite[\S 2.3, \S 4 and \S 5]{KM98}) for more details on singularities.

Let $X$ be a quasi-projective  surface with klt singularities and let $p:X\to Z$ be a projective morphism. Then we can run a   minimal model program (MMP for short) for $X$ over $Z$. We obtain a sequence of extremal contractions  over $Z$, $$X=X_0 \to X_1 \to \cdots \to X_n $$ such that either $K_{X_n}$ is nef over $Z$ or there is a Mori fibration $X_n \to B$ over $Z$. We will call  $X_n$  the result of this MMP.  For more details on MMP, see, \textit{e.g.}, \cite[\S 1.4 and \S 3.7]{KM98}.

The following lemma gives an estimation on intersection numbers in a projective surface. 

\begin{lemma}
\label{intersection-smooth-point}
Let $X$ be a normal projective $\mathbb{Q}$-factorial surface. Let $C$, $D$ be two different projective integral curves in $X$. Assume that both $C$ and $D$ pass through a smooth point $x$ of $X$. Then $C\cdot D \geqslant 1$.
\end{lemma}

\begin{proof}
Let $r:\widetilde{X} \to X$ be the minimal resolution of $X$. Then by the projection formula, we obtain $$C \cdot D = r^*C \cdot (r^{-1})_*D \geqslant  (r^{-1})_*C \cdot (r^{-1})_*D.$$ Since $r$ is minimal, and $X$ is smooth at $x$, $r^{-1}$ is an isomorphism around $x$. Hence both $(r^{-1})_*C$ and $(r^{-1})_*D$ pass through the smooth point $y$, where $\{y\}=r^{-1}(\{x\})$. Since $\widetilde{X}$ is smooth, we have $(r^{-1})_*C \cdot (r^{-1})_*D \geqslant 1$ by \cite[Prop. V.1.4]{Har77}. Hence  $C\cdot D \geqslant 1$.
\end{proof}

Let $X$ be a complex Fano surface with canonical singularities such that the smooth locus of $X$ is simply connected. Assume that $X$ has Picard number $1$. Then there are at most two singular points in $X$ (See \cite{MZ88}). If $X$ is smooth, then $X \cong \p^2$. If there is one singular point, then possible types of singularities of this point are $$A_1,\ A_4,\ D_5,\ E_6,\ E_7,\ E_8.$$ If there are two singular points, then one is of type $A_1$, and the other is of type $A_2$. For more details on $ADE$ singularities, see for example \cite[\S 5.2]{KM98}. Except the case when there is a singular point of type $E_8$, for each type of singularities, there is exactly one isomorphic class for the surface. If the surface has a singular point of type $E_8$, then there are exactly two  isomorphic classes   $S^c(E_8)$ (See \cite[Lem. 3.6.(1)]{KM99})  and $S^n(E_8)$ (See \cite[Lem. 3.6.(2)]{KM99}) of such surfaces.

We recall some properties of the surfaces $S^c(E_8)$ and $S^n(E_8)$. These two surfaces can be constructed as follows. Choose a singular cubic rational curve $C$ in $X_1=\p^2$. Let $x$ be one of the  smooth inflection points of $C$. Recall that $x$ is an inflection point if the line $L$ in $X_1$ tangent to $C$ at point $x$ meets $C$   at $x$ with order at least $3$ (See \cite[\S IV.1]{EH00} for more details).  Starting from $X_1=\p^2$, we blow up the point $x$, and we get a surface $X_2$. We construct a rational surface $X_9$ by induction as follows. Assume that $X_i$ is constructed ($i\geqslant 2$). Then $X_{i+1}\to X_i$ is the blow-up of the intersection point of the strict transform of $C$ in $X_i$ and the exceptional divisor of $X_i\to X_{i-1}$. There are exactly eight $(-2)$-curves in $X_9$: the strict transform of $L$, and the strict transforms of all the exceptional curves of $X_8\to X_1$.  Let $X_9\to X$ be the blow-down of all $(-2)$-curves in $X_9$. Then $X$ is isomorphic to $S^c(E_8)$ if $C$ is a cuspidal rational curve, or is isomorphic to  $S^n(E_8)$ if $C$ is a nodal curve. (This is why we use the terminology of these two surfaces. The index ``$c$'' is for cuspidal, and the index ``$n$'' is for nodal) The linear system $|-K_X|$ has dimension $1$, and has a unique basepoint  which is a smooth point of $X$. A  general member  of this linear system is a smooth elliptic curve in $X$. If $Y\to X$ is the blow up of the basepoint, then there is a fibration $p_Y:Y\to \p^1$, induced by $|-K_X|$, whose general fibers are elliptic curves. There is a unique  rational curve    in $|-K_X|$ which passes through the singular point of $X$. If $X\cong S^c(E_8)$ then there exist exactly two rational curves in $|-K_X|$. The one contained in the smooth locus of $X$ is   the strict transform of $C$ in $X$. If $X \cong S^n(E_8)$, then there are exactly three rational curves in $|-K_X|$. Both of the two rational curves contained in the smooth locus of $X$ are singular rational curves with a node (one of them is the strict transform of $C$ in $X$).

\begin{center}
\begin{tikzpicture}
\draw (-1,1) -- (-5,1);
\filldraw (-4,1) circle  node[above] {Nodal curve};
\draw (-2,2) -- (-2,-2);
\draw (-1,2) -- (-4,-1);
\filldraw (-4,-1) circle  node[below] {Nodal curve};
\filldraw (-2,1) circle (2pt) node[below right] {Basepoint};
\filldraw (-2,-1) circle (2pt) node[below right] {Singular point};
\filldraw (-3,-2) circle  node[below] {$X\cong S^n(E_8)$};
\draw (3,1) -- (7,1);
\filldraw (4,1) circle  node[above] {Cuspidal curve};
\draw (6,2) -- (6,-2);
\filldraw (6,1) circle (2pt) node[below right] {Basepoint};
\filldraw (6,-1) circle (2pt) node[below right] {Singular point};
\filldraw (5,-2) circle  node[below] {$X\cong S^c(E_8)$};
\filldraw (2,-3) circle  node[below] {Rational curves in $|-K_X|$};
\end{tikzpicture}
\end{center}


\subsection{Simple connectedness}
We collect some properties on the algebraic simple connectedness of projective surfaces. A variety $X$ is called algebraically simply connected if  any \'etale morphism $\tau: X'\to X$, with $X'$  irreducible,  is an isomorphism.

\begin{lemma}
\label{sim-con-open}
Let $X$ be a variety. Let $U$ be a non-empty Zariski open subset of $X$. If $U$ is algebraically simply connected, then so is $X$.
\end{lemma}

\begin{proof}
Let $c:X' \to X$ be an \'etale cover such that  $X'$ is irreducible. Let $U'=c^{-1}(U)$. Then  $U'$ is irreducible. Hence, $c|_{U'}$ is an isomorphism for $U$ is algebraically simply connected. This implies that $c$ is an isomorphism. Thus $X$ is algebraically simply connected.
\end{proof}

\begin{lemma}[{\cite[Lem. 3.3]{KM99}}]
\label{sim-con-mmp}
Let $X$ be a quasi-projective surface with  canonical singularities. Let $f:X\to Y$ a divisorial contraction in the MMP. Assume that the smooth locus of $X$ is algebraically simply connected. Then the smooth locus of $Y$ is algebraically simply connected.
\end{lemma}

\begin{proof}
Let $r:\widetilde{X} \to X$ and $s:\widetilde{Y} \to Y$ be minimal resolutions of $X$ and $Y$. Then there is a natural morphism $\tilde{f}:\widetilde{X} \to \widetilde{Y}$. Moreover, $\widetilde{X}$ can be obtained from $\widetilde{Y}$ by a sequence of blow-ups.

Let $E$ (\textit{resp.} $D$) be the exceptional set of $r$ (\textit{resp.} of $s$). Then $\widetilde{X}\backslash E$ is algebraically simply connected, and we only have to prove that $\widetilde{Y}\backslash D$ is algebraically simply connected. Since $f^{-1}$ is an isomorphism around the singular points of $Y$ (See \cite[Prop. 5.4]{Ou14a}), we have $\tilde{f}^{-1}(D)\subseteq E$. Hence, $\widetilde{X}\backslash E \subseteq \tilde{f}^{-1}(\widetilde{Y}\backslash D)$. By Lemma \ref{sim-con-open}, we obtain that $\tilde{f}^{-1}(\widetilde{Y}\backslash D)$ is algebraically simply connected. Note that $\tilde{f}^{-1}(\widetilde{Y}\backslash D)$ can be obtained from  $\widetilde{Y}\backslash D$ by a sequence of blow-ups. Hence, $\widetilde{Y}\backslash D$ is also algebraically simply connected.
\end{proof}

\subsection{Construction of an intermediate surface}

Consider a singular klt Fano surface $X$ whose smooth locus  is algebraically simply connected. In this subsection, we will construct a  surface $Z$ which is birational to $X$ and  with many nice properties. More precisely, we will prove the following proposition.

\begin{prop}
\label{aux-surface}
Let $X$ be a singular rational surface with klt singularities whose smooth locus is algebraically simply connected. Then there is a rational surface $Z$ with klt singularities such that
\begin{enumerate}
\item[(1)] There is a fibration $p:Z\to \p^1$ whose general fibers are smooth rational curves.
\item[(2)] There is a birational morphism $\pi:Z\to X$ such that every $\pi$-exceptional curve is horizontal over $\p^1$ with respect to $p$.
\item[(3)] The minimal resolution $\widetilde{X}\to X$ factors through $\widetilde{X} \to Z$.
\item[(4)] There is at most one  $\pi$-exceptional curve which is crepant over $X$, and if there is such a curve, then it is a section of $p$ over $\p^1$.
\end{enumerate}

\centerline{
\xymatrix{
\widetilde{X}   \ar[d]_{\mathrm{minimal \ resolution}} \ar[dr]^{\mathrm{minimal \ resolution}}
& 
\\
Z  \ar[r]^{\pi} \ar[d]_{p}
&X  
\\ 
\p^1 
& 
}
}

\end{prop}

This surface $Z$ will be very useful for the proof of Theorem \ref{thm-rat-surface}.  For the construction, we will need two lemmas.

\begin{lemma}
\label{contract-sub-graph}
Let $Y$ be a klt quasi-projective surface. Let $E_i \subseteq Y$  $(i=1,...,r)$ be smooth projective irreducible $K_Y$-non-negative rational curves forming an snc divisor $\bigcup_{i=1}^r E_i$.  Assume that the intersection matrix $(E_i \cdot E_j)_{1\leqslant i,j \leqslant r}$ is  negative definite, and that there are rational numbers $-1<a_i \leqslant 0$ such that $(K_Y- \sum_{i=1}^r a_iE_i) \cdot E_j=0$ for every $j$. Let $\{C_k\}$ be any subset of $\{E_i\}$. Then there is a birational morphism $Y \to X$ contracting the $C_k$'s and no other curves. Moreover, $X$ has klt singularities.
\end{lemma}

\begin{proof}
We may assume that $\{C_k\}$ is not empty. By renumbering if necessary, we may assume that $\{C_k\}=\{E_1,...,E_s\}$ with $1\leqslant s \leqslant r$. Note that the intersection matrix $(E_i\cdot E_j)_{1\leqslant i,j \leqslant s}$ is negative definite. Hence there are real numbers $b_i$  such that $(K_Y- \sum_{i=1}^{s} b_iE_i) \cdot E_j=0$ for all $1 \leqslant j \leqslant s$. Since $K_Y \cdot E_i$ is an integer for all  $i$ and $E_i\cdot E_j$ is integer for all $1\leqslant i,j \leqslant s$, the $b_i$'s are rational numbers. On the one hand, since $K_Y \cdot E_i \geqslant 0$ for all $i$, we have $b_i \leqslant 0$ for all $1 \leqslant i \leqslant s$ by \cite[Cor. 4.2]{KM98}. On the other hand, we have that $(K_Y- \sum_{i=1}^{s} a_iE_i) \cdot E_j= \sum_{i=s+1}^{r}a_iE_i \cdot E_j \geqslant 0$ for all $1 \leqslant j \leqslant s$. Hence by \cite[Cor. 4.2]{KM98}, we have $-1<a_i \leqslant b_i$ for all $1\leqslant i \leqslant s$. Note that the pair $(Y,-\sum_{i=1}^{s} b_iE_i)$ is klt. Hence by the same argument as in \cite[Prop. 4.10]{KM98}, there is a birational morphism $g:Y \to X$ contracting the $C_k$'s and no other curves. Moreover, $X$ has klt singularities since the pair $(Y,-\sum_{i=1}^{s} b_iE_i)$ is klt and $g^*K_X=K_Y-\sum_{i=1}^{s} b_iE_i$.
\end{proof}

\begin{rem}
\label{rem-contract-sub-graph}
One of the application of Lemma \ref{contract-sub-graph} is as follows. Let $g:Y\to Z$ be a partial minimal resolution of singularities of a klt surface $Z$. That is, if $z\in Z$ is a point such that $g^{-1}$ is not an isomorphism around $z$, then $z$ is singular, and $g$ is the minimal resolution of the singularity at $z$. Let $\{E_i\}$ be the set of the $g$-exceptional curves. Since $Z$ has klt singularities, the $E_i$'s satisfy the conditions in the lemma. Hence for any subset $\{C_k\}$ of $\{E_i\}$,  there is a birational morphism $Y \to X$ contracting exactly the $C_k$'s. The morphism $g:Y\to Z$ factors through $Y\to X$. Moreover, if $Z$ has canonical singularities, then so is $X$ (in this case $a_i=0$ for all $i$, and $b_j=0$ for all $j$ by the calculation in the proof of the lemma).
\end{rem}

In the following lemma, we prove Proposition \ref{aux-surface} under the assumption that $X$ has canonical singularities. For the proof, we use the classification  of Fano surfaces with canonical singularities which have Picard number $1$ (See \cite{MZ88} or \cite[\S 3]{KM99}).

\begin{lemma}
\label{extract-fib}
Let $X$ be a singular rational surface with canonical singularities. Assume  that the smooth locus of $X$ is algebraically simply connected. Let $r:\widetilde{X} \to X$ be the minimal resolution of singularities. Then there is a rational surface $X_1 \overset{g}{\longrightarrow} X$ with canonical singularities such that:
\begin{enumerate}
\item The minimal resolution $\widetilde{X} \to X$ factors through $\widetilde{X} \to  X_1$ which is also the minimal resolution of $X_1$.

\item The morphism $g$ contracts at most one curve.

\item There is a fibration $p_1:X_1 \to \p^1$ whose general fibers are smooth rational curves.

\item If $g$ contracts a curve $C_1$, then $C_1$  is a section of $p_1$ over $\p^1$, and the strict transform of $C_1$ in $\widetilde{X}$ is a $(-2)-$curve.
\end{enumerate}

\centerline{
\xymatrix{
\widetilde{X}   \ar[d]_{\mathrm{minimal \ resolution}} \ar[dr]^{\mathrm{minimal \ resolution}}
& 
\\
X_1  \ar[r]^g_{\mathrm{crepant}} \ar[d]_{p_1}
&X  
\\ 
\p^1 
& 
}
}

\end{lemma}

\begin{proof}
Let $f:X\to X'$ be the result of a MMP for $X$. Then $X'$ has   canonical singularities, and the smooth locus of $X'$ is algebraically simply connected by Lemma \ref{sim-con-mmp}. If $X'$ has Picard number at least $2$, then there is a Mori fibration $X'\to \p^1$ which induces a fibration $X\to \p^1$. In this case, we let $g:X_1 \to X$ be the identity map. In the following, we will assume that $X'$ has Picard number $1$.

If $X'$ is not smooth, then by \cite[Lem. 3.8]{KM99}, there is a rational surface $X'_1 \overset{g'}{\longrightarrow} X'$ with canonical singularities such that

\begin{enumerate}

\item The minimal resolution $\widetilde{X}' \to X'$ factors through $\widetilde{X}' \to X_1'$ which is also the minimal resolution of $X_1'$.

\item The morphism $g'$ contracts exactly one curve $C'_1$.

\item There is a fibration $p_1':X_1 '\to \p^1$ whose general fibers are smooth rational curves.

\item The curve $C'_1$ is a section of $p_1'$ over $\p^1$, and the strict transform of $C'_1$ in $\widetilde{X}'$ is a $(-2)-$curve.

\end{enumerate}

There is a natural morphism $\widetilde{X} \to \widetilde{X}'$. Contract all exceptional curves of $\widetilde{X} \to X$, except maybe the strict transform of $C'_1$ in $\widetilde{X}$, we obtain  birational morphisms $\widetilde{X} \longrightarrow X_1 \overset{g}{\longrightarrow} X$ (See Remark \ref{rem-contract-sub-graph}). Then $X_1$ has   canonical singularities, and there is a natural morphism $f_1:X_1 \to X_1'$. The composition $p_1=p_1'\circ f_1$ gives a fibration from $X_1$ to $\p^1$. Hence $X_1$ satisfies the  conditions in the lemma.

If $X'$ is smooth, then $X' \cong \p^2$. Let $\theta :Y \to X'$ be the last step of the MMP for $X$. Then there is a natural morphism $\phi:X \to Y$. Let $s:\widetilde{Y} \to Y$ be the minimal resolution of $Y$. Then $\widetilde{Y}$ can be obtained from $X'$ (which is isomorphic to $\p^2$) by a sequence of blow-ups. Since $X$ is not smooth, $X$ is different from $X'$. Hence $Y\neq X'$, and $\widetilde{Y} \neq X'$. Moreover, since $\theta^{-1}: X' \dashrightarrow Y$ is an isomorphism outside exactly one point $x'$ of $X'$, every exceptional divisor of 
$\widetilde{Y} \to X'$ is over the point $x'$. Blow up the point $x'$ in $X'$, we obtain a surface $Z$. Then $\widetilde{Y} \to X'$ factorise through $\widetilde{Y} \to Z$, and there is a fibration from $Z$ to $\p^1$. Hence we obtain a fibration $\widetilde{Y} \to \p^1$ such that there is at most one curve in $\widetilde{Y}$ which is both exceptional for $s:\widetilde{Y} \to Y$ and horizontal over $\p^1$. Moreover, if this curve exists, then it is a section over $\p^1$ (it can only be the strict transform of the exceptional divisor of $Z\to X'$). Contract all of the $s$-exceptional curves in $\widetilde{Y}$ which are not horizontal over $\p^1$ (See Remark \ref{rem-contract-sub-graph}), we get a rational surface $Y_1\overset{h}{\longrightarrow} Y$ which has at most canonical singularities such that 

\begin{enumerate}

\item The minimal resolution $\widetilde{Y} \to Y$ factors through $\widetilde{Y} \to  Y_1$ which is also the minimal resolution of $Y_1$.

\item The morphism $h$ contracts at most one curve.

\item There is a fibration $q_1:Y_1 \to \p^1$ whose general fibers are smooth rational curves.

\item If $h$ contracts a curve $B_1$, then $B_1$  is a section of $q_1$ over $\p^1$, and the strict transform of $B_1$ in $\widetilde{Y}$ is a $(-2)-$curve.

\end{enumerate}

There is a natural morphism $\widetilde{X} \to \widetilde{Y}$. By contracting all of the exceptional curves of $\widetilde{X} \to X$, except maybe the strict transform of $B_1$ in $\widetilde{X}$, we obtain birational morphisms $\widetilde{X} \longrightarrow X_1 \overset{g}{\longrightarrow} X$ (See Remark \ref{rem-contract-sub-graph}). Then $X_1$ has   canonical singularities, and there is a natural morphism $\phi_1:X_1 \to Y_1$. The composition $p_1=q_1\circ \phi_1$ gives a fibration from $X_1$ to $\p^1$. Hence $X_1$ satisfies the conditions in the lemma.
\end{proof}

Now we can prove Proposition \ref{aux-surface}.

\begin{proof}[{Proof of Proposition \ref{aux-surface}}]
Since $X$ has isolated singularities, there is a birational morphism $h: X_1\to X$ which resolves exactly the non-canonical singularities of $X$. We may also assume that this partial resolution is minimal. Then $X_1$ has at most canonical singularities, and none of the $h$-exceptional curves is crepant over $X$ (See \cite[Cor. 4.3]{KM98}). Moreover, if $\widetilde{X} \to X$ is the minimal resolution, then there is a natural morphism $\widetilde{X} \to X_1$ which is also the minimal resolution of $X_1$. Since the smooth locus of $X$, which is algebraically simply connected, is isomorphic to an open subset of the smooth locus of $X_1$, the smooth locus of $X_1$ is algebraically simply connected by Lemma \ref{sim-con-open}. 

First we assume that $X_1$ is singular. Then by Lemma \ref{extract-fib}, there is a rational surface $X_2\overset{g}{\longrightarrow} X_1$ with canonical singularities such that:
\begin{enumerate}

\item The minimal resolution $\widetilde{X} \to X_1$ factors through $\widetilde{X} \to  X_2$ which is also the minimal resolution of $X_2$.

\item The morphism $g$ contracts at most one curve.

\item There is a fibration $p_2:X_2 \to \p^1$ whose general fibers are smooth rational curves.

\item If $g$ contracts a curve $C_1$, then $C_1$  is a section of $p_2$ over $\p^1$, and the strict transform of $C_1$ in $\widetilde{X}$ is a $(-2)-$curve.

\end{enumerate}
Let $\phi:X_2 \to X$ be the composition of $X_2\overset{g}{\longrightarrow} X_1 \overset{h}{\longrightarrow} X$. Since $g$ contracts at most one curve, and none of the $h$-exceptional curves is crepant over $X$, there is at most one curve in $X_2$ which is  crepant over $X$. By Remark \ref{rem-contract-sub-graph}, we can contract all curves in $X_2$ which are both contracted by $\phi$ and by $p_2$. We obtain a normal surface $Z$ such that $\widetilde{X}\to Z$ is the minimal resolution. Then $Z$ has klt singularities. There is a natural fibration $p:Z\to \p^1$ induced by $p_2$. Moreover there is a natural birational morphism $\pi:Z \to X$ induced by $\phi$.

\centerline{
\xymatrix{
\widetilde{X}   \ar[d]_{\mathrm{minimal \ resolution}} \ar[dr]^{\mathrm{minimal \ resolution}}
& 
\\
X_2 \ar[d] \ar@/_2pc/[dd]_{p_2} \ar[r]^g_{\mathrm{crepant}} & X_1 \ar[d]_h^{\mathrm{partial\ resolution}}
\\
Z  \ar[r]^{\pi} \ar[d]^{p}
&X  
\\ 
\p^1 
& 
}
}

Then there is at most one $\pi$-exceptional curve which is crepant over $X$ since there is at most one curve in $X_2$ which is  crepant over $X$. If this curve exists, then it is the strict transform of $C_1$ in $Z$. Hence it is a section over $\p^1$.

Now we assume that $X_1$ is a smooth rational surface. Then $X_1$ is different from $\p^2$ since $X$ is singular. Hence, there is a fibration $p_1:X_1\to \p^1$. Let $g:X_2\to X_1$ be the identity morphism and let $p_2=p_1\circ g:X_2\to \p^1$. Let $\phi=h\circ g: X_2\to X$. We can construct $\pi:Z \to X$ as before such that $Z$ satisfies all the conditions in the proposition. Moreover, in this case,  none of the $\pi$-exceptional curves is crepant over $X$.
\end{proof}

\subsection{$\p^1$-fibration over a curve}

In this subsection, we will give some properties of fibered surfaces $X\to B$ such that $B$ is a smooth curve, and general fibers of the fibration are smooth rational curves. The fibration $X\to B$ is called a $\p^1$-bundle if it is smooth.

\begin{lemma}
\label{family-smooth-curves}
Let $p:X \to B$ be a projective fibration from a normal quasi-projective surface to a smooth curve. Assume that every fiber of $p$ is reduced and irreducible, and there is a smooth fiber isomorphic to $\p^1$. Then $p$ is a  $\p^1$-bundle.
\end{lemma}

\begin{proof}
Since $B$ is a smooth curve, $p$ is a flat morphism. Hence all fibers of $p$ have the same arithmetic genus by \cite[Cor. III.9.10]{Har77}. By assumption, there is a point $b$ in $B$ such that the fiber of $p$ over $b$ is a smooth rational curve. Since every fiber of $p$ is reduced and irreducible, we obtain that every fiber of $p$ is a curve with arithmetic genus $0$. Thus, every fiber of $p$ is isomorphic to $\p^1$. Since $p$ is flat, and every fiber of $p$ is smooth, $p$ is a smooth morphism. Hence  $p$ is a  $\p^1$-bundle.
\end{proof}

Recall that a Hirzebruch surface $\Sigma_k$ is a smooth rational ruled surface isomorphic to $\p_{\p^1}(\sO_{\p^1}\oplus \sO_{\p^1}(-k))$. In particular, if $k>0$, then there is a unique rational curve $C$ in $\Sigma_k$ such that $C^2=-k$ (See \cite[\S V.2]{Har77}).

\begin{lemma}
\label{can-degree-formula}
Let $\widetilde{X}=\Sigma_k$ be a Hirzebruch surface with $k>0$, and let $C$ be the unique rational curve  in $\widetilde{X}$ such that $C\cdot C=-k$. Let $p:\widetilde{X} \to \p^1$ be the $\p^1$-bundle, and let $F$ be a fiber of $p$.  Let $r:\widetilde{X} \to X$ be the morphism contracting $C$. Then $K_X \cdot K_X=8+\frac{(k-2)^2}{k}$, $-K_X\cdot \pi_*F =1+\frac{2}{k}$, and $\pi_*F \cdot \pi_*F= \frac{1}{k}$. In particular, $-K_X$ cannot be a minimal curve.
\end{lemma}

\begin{proof}
We have $-K_{\widetilde{X}}=-\pi^*K_X+\frac{k-2}{k}C$, $K_{\widetilde{X}} \cdot K_{\widetilde{X}}=8$, and $C \cdot F =1$. Hence $K_X \cdot K_X=8+\frac{(k-2)^2}{k}$, and $-K_X\cdot \pi_*F =1+\frac{2}{k}$. We also obtain that $\pi_*F \cdot \pi_*F= \frac{1}{k}$.
\end{proof}

In the remaining of this subsection, our aim is to prove the following proposition.

\begin{prop}
\label{double-fibre-canonical}
Let $p:X\to B$ be Mori fibration from a klt quasi-projective  surface to a smooth curve. Let $b$ be a point in $B$, and let $C=\mathrm{Supp}(p^*b)$. Assume that $p^*b=2C$. Then $X$ has canonical singularities along $p^*b$. More precisely, we have exactly two possibilities
\begin{enumerate}
\item  there are two singular point on $C$, and both of them are of type $A_1$;
\item   there is one singular point on $C$, and the singularity is of type $D_i$ for some $i\geqslant 3$.
\end{enumerate}
\end{prop}

We will first prove several lemmas.

\begin{lemma}
\label{snc-fiber-support}
Let $p:X\to B$ be a projective fibration from a smooth quasi-projective surface to a smooth curve $B$ such that general fibers of $p$ are smooth rational curves. Let $b$ be a point in $B$. Then the divisor $p^*b$ has snc support. Moreover, the dual graph of the support of $p^*b$ is a tree.
\end{lemma}

\begin{proof}
Let $X\to X'$ be the result of a $p$-relative MMP. Then $X'$ is smooth, and there is a natural morphism $p':X\to B$. Since general fibers of $p$ are rational curves, $p'$ is a Mori fibration. Hence $p'^*b$ is a prime divisor since $X'$ is smooth (See \cite[Lem. 3.4]{KM99}). Since $X$ can be obtained by a sequence of blow-ups from $X'$, the divisor $p^*b$ has snc support, and the dual graph of the support of $p^*b$ is a tree.
\end{proof}

\begin{lemma}
\label{exactly-one--1-curve}
Let $p:X\to B$ be a projective fibration from a smooth quasi-projective surface to a smooth curve $B$ such that general fibers of $p$ are smooth rational curves. Let $b$ be a point in $B$. Let $(p^*b)_{\mathrm{red}}$ be the sum of the components of $p^*b$.  Assume that there is exactly one $(-1)$-curve $C$ in $p^*b$, then the following properties hold.
\begin{enumerate}
\item[(1)] One of the curves in $p^*b$ which meet $C$ is a $(-2)$-curve.
\item[(2)] Assume that there is a chain of rational curves $E=\sum_{i=1}^{s} C_i$ in the fiber $p^*b$ such that the dual graph of $C+( \sum_{i=1}^{s} C_i)$ is as follows. 

\centerline{
\xymatrix{
\overset{C}{\bullet} \ar@{-}[r] &\overset{C_1}{\bullet} \ar@{-}[r]  &\overset{C_2}{\bullet} \ \ \cdots \ \ \overset{C_s}{\bullet}
}}
Assume further that $E+C$ and $((p^*b)_{\mathrm{red}}-E-C)$ intersect only along $C_s$. Then all of the $C_i$'s are $(-2)$-curves.

\item[(3)] The multiplicity of $C$ in $p^*b$ is larger than one.
\end{enumerate}
\end{lemma}

\begin{proof}
If $p^*b$ has one component, then it is a $0$-curve. If $p^*b$ has two components, then both of them are $(-1)$-curves. Since there is exactly one $(-1)$-curve in $p^*b$ by assumption, we obtain that  $p^*b$ has at least three components.

(1) Let $X\to Y$ be the birational morphism which contracts exactly $C$. Then $Y$ is smooth, and there is a fibration $q:Y\to B$ induced by $p$. Note that $q^*b$ has at least two components. Hence, from the MMP, we know that there is at least one $(-1)$-curve $D$ in $q^*b$. Since there is exactly one $(-1)$-curve $C$ in $p^*b$, the strict transform of $D$ in  $X$ meets $C$, and is a $(-2)$-curve.

(2) We will prove by induction on $i$. If $i=1$, then from (1), we know that $C_1$ is a $(-2)$-curve.  Assume that $C_1, ..., C_j$ are all $(-2)$-curves for some $j\geqslant 1$. We will prove that $C_{j+1}$ is a $(-2)$-curve. Set $C_0=C$. There is a birational morphism $h:X\to X'$ which contracts exactly $C_0,...,C_{j-1}$. Let $p':X'\to \p^1$ be the fibration induced by $p$. Then the strict transform of $C_j$ in $X'$ becomes the unique $(-1)$-curve in $p'^*b$ since $E$ and $((p^*b)_{\mathrm{red}}-E)$ intersect only along $C_s$. By (1), we obtain that the strict transform of $C_{j+1}$ in $X'$ is a $(-2)$-curve.  This shows that $C_{j+1}$ is  a  $(-2)$-curve since $h$ is an isomorphism around $C_{j+1}$. This completes the induction.

(3) We will prove  by induction on the number  $k$ of components of the fiber $p^*b$. If $k=3$, then the multiplicity of $C$ is $2$. Assume that the lemma is true for $k=l$, where $l\geqslant 3$ is an integer. 

Now we assume that $k=l+1$. Let $X\to Y$ be the birational morphism which contracts exactly $C$. Then $X$ is the blow-up of a point $y$ in $Y$. Let $q:Y\to \p^1$ be the fibration induced by $p$. If there are more than one $(-1)$-curves in $q^*b$, then they intersect at the point $y$. In this case, the multiplicity of $C$ is larger than one. If there is exactly one $(-1)$-curve $D$ in $q^*b$, then the multiplicity $m$ of $D$ in $q^*b$ is larger than one. In this case, the multiplicity of $C$ in $p^*b$ is not less than $m$. This completes the induction and the proof of the lemma.
\end{proof}

\begin{lemma}
\label{fiber-self-intersection-non-positive}
Let $p:X \to B$ be a projective fibration from a smooth surface to a smooth curve. Let $b$ be a point in $B$. Let $C_1,...,C_r$ be the components of $p^*b$. Then the intersection matrix $(C_i\cdot C_j)$ is negative. Moreover, $(\sum a_iC_i)^2=0$ if and only if $\sum a_iC_i$ is proportional to the fiber. In particular, if $C$ is a component  of $p^*b$ such that  $C \cdot C=0$, then $p^*b$ is irreducible with support $C$.
\end{lemma}

\begin{proof}
See \cite[\S A.7]{Reid97}.
\end{proof}

\begin{lemma}
\label{sing-points-two}
Let $p:X\to B$ be a Mori fibration from a klt quasi-projective surface to a smooth curve $B$.  Let $b$ be a point in $B$. Assume that $X$ is not smooth  along $p^*b$. Then there are at most $2$ singular points of $X$ on $p^*b$.
\end{lemma}

\begin{proof}
Let $r:\widetilde{X} \to X$ be the minimal resolution of $X$. Let $\tilde{p}=p\circ r$. Let $D$ be the support of $p^*b$. Since $p$ is a Mori fibration, $D$ is irreducible. Let  $\widetilde{D}$ be the strict transform of $D$ in $\widetilde{X}$. Then $\widetilde{D}$ is the only $K_{\widetilde{X}}$-negative curve in the fiber $\tilde{p}^*b$. In fact, $\widetilde{D}$ is a $(-1)$-curve by Lemma \ref{snc-fiber-support} and Lemma \ref{fiber-self-intersection-non-positive}. Let $\widetilde{X}\to S$ be the divisorial contraction which contracts $\widetilde{D}$. Let $q:S \to \p^1$ be the  fibration induced by $\tilde{p}$. Then $S$ is smooth, and $q^*b$ has snc support by Lemma \ref{snc-fiber-support}. This implies that there are at most $2$ singular points of $X$ on $p^*b$.
\end{proof}

\begin{lemma}
\label{double-fibre-canonical-2-sing}
Let $p:X\to B$ be Mori fibration from a klt quasi-projective  surface to a smooth curve. Let $b$ be a point in $B$, and let $C=\mathrm{Supp}(p^*b)$. Assume that $p^*b=2C$, and that there are two singular points of $X$ on $C$. Then $X$ has canonical singularities along $C$, and  the two singular points are   of type $A_1$.
\end{lemma}

\begin{proof}
Let $r:\widetilde{X}\to X$ be the minimal resolution and $\tilde{p}:\widetilde{X}\to B$ be the fibration  induced by $p$. If $\widetilde{C}$ is the strict transform of $C$ in $\widetilde{X}$, then it is the unique $(-1)$-curve in $\tilde{p}^*b$. Let $\widetilde{X}\to Y$ be the morphism which contracts exactly $\widetilde{C}$. Let $q:Y\to B$ be the fibration induced by $\tilde{p}$.

Since there are two singular points of $X$ on $C$, and the multiplicity of the fiber $p^*b$ is two, we can decompose $\tilde{p}^*b-2\widetilde{C}$ as $D+D'+R$ with $D$, $D'$ irreducible such that $\widetilde{C} \cdot D = \widetilde{C} \cdot D'=1$, $D \cdot D' =0$, and $\widetilde{C} \cdot R =0$. By   Lemma \ref{exactly-one--1-curve}.(1), either $D$ or $D'$ is a $(-2)$-curve. Without loss of generality, we may assume that $D$ is a $(-2)$-curve. Then $$0=\tilde{p}^*b \cdot D = 2\widetilde{C} \cdot D + D^2 +D' \cdot D + R \cdot D = R\cdot D.$$ In particular, the intersection of $R$ and $D$ is empty.

Let $Y\to Z$ be the morphism contracting exactly the strict transform of $D$ and $s:Z\to B$ be the fibration induced by $q$. Let $E$  be  the strict transform of $D'$ in $Z$. Then  there is at least one component in $s^*b$ which has negative intersection number with $-K_Z$. That is, there is at least one component in $s^*b$ which has self-intersection number larger than $-2$. Since $R$ does not meet $D$, we obtain that  $E$  is the unique curve in $s^*b$ which can have self intersection number larger than $-2$. Hence $E^2\geqslant -1$. Since the multiplicity of $E$ in $s^*b$ is one, by Lemma \ref{exactly-one--1-curve}.(3), we  have  $E^2\neq -1$. However, by Lemma \ref{fiber-self-intersection-non-positive}, we  have $E^2\leqslant 0$. Hence, $E^2=0$ and $D'$ is a $(-2)$-curve. For the same reason as before, we obtain that $R\cdot D'=0$.

Hence, $R=0$ and $\tilde{p}^*b=D+D'+2\widetilde{C}$. This proves the lemma.
\end{proof}

Now we will prove  Proposition \ref{double-fibre-canonical}.

\begin{proof}[{Proof of Proposition \ref{double-fibre-canonical}}]
Since $p$ is a Mori fibration, $C$ is irreducible. Note that $X$ is not smooth along $C$ since $p^*b=2C$. In fact, assume that it is. Then by the adjunction formula, we have $$2h^1(C,\sO_C)-2=(K_X+C)\cdot C = K_X\cdot C = -1.$$ This is a contradiction.

Let $t$ be the number of singular points of $X$ on the support of $p^*b$. Then $t\leqslant 2$ by Lemma \ref{sing-points-two}. Let $r:\widetilde{X} \to X$ be the minimal resolution, and let $\widetilde{C}$ be the strict transform of $C$ in $\widetilde{X}$. Then $\widetilde{C}$ is a $(-1)$-curve. Let $\tilde{p}=p\circ r$.  Let $E=\tilde{p}^*b-2\widetilde{C}$.

First assume that there are exactly two singular points on $C$. Then the singularities at these two points are canonical of type $A_1$ by Lemma \ref{double-fibre-canonical-2-sing}.

Now assume that there is only one singular point on $C$. Let $C_0=\widetilde{C}$, and let $E_0=E$.   We will show that there is a positive integer $i$ such that we can decompose $$ \tilde{p}^*b=2(C_0+C_1+\cdots +C_i)+D+D'+R \ \ \ \ \ \ \ \ \ \  (*)$$ such that all of the $C_j$'s with $j>0$ are $(-2)$-curves, that both $D$ and $D'$ are smooth rational curves and that $R$ is disjoint from $C_0+\cdots +C_i$. Furthermore, we have $D \cdot D'=0$, $C_i \cdot D= C_i \cdot D'=1$, $C_j \cdot C_{j+1}=1$ for $1 \leqslant j \leqslant i-1$, and $C_j \cdot C_k=0$ if $k-j>1$. In particular, the dual graph of $\tilde{p}^*b -R$ is as follows.

\centerline{
\xymatrix  @R=.2pc {
 & \overset{D'}{\bullet} \ar@{-}[d]\\
\underset{D}{\bullet}  \ar@{-}[r] & \underset{C_i}{\bullet} \ar@{-}[r] &  \underset{C_{i-1}}{\bullet} \ \cdots \ \   \underset{C_1}{\bullet}\ar@{-}[r] & \underset{C_0}{\bullet}
}}

We will  construct the decomposition $(*)$ by induction. Since there is only one singular point over $C$, and $\tilde{p}^*b$ is an snc divisor, the support of $E_0$ intersects $C_0$ transversally at a point. Moreover since $C_0\cdot E_0=2$, we can decompose $E_0$ into $2C_1+E_1$, where $C_1$ is a smooth rational curve and $C_0\cdot E_1=0$. From Lemma \ref{exactly-one--1-curve}.(1), we know that $C_1$ is a $(-2)$-curve.  We have $$C_1\cdot E_0=C_1\cdot (\tilde{p}^*b-2C_0)=-2C_1\cdot C_0=-2  \ \mathrm{and}\   E_0^2=(\tilde{p}^*b-2C_0)^2= 4C_0^2=-4.$$ Hence we obtain that $$E_1^2=(E_0-2C_1)^2=-4\ \mathrm{and}\ C_1 \cdot E_1=C_1\cdot(E_0-2C_1)=2.$$

Assume that we can decompose $$\tilde{p}^*b=2(C_0+C_1+\cdots +C_k)+E_k$$ for some $k> 0$ such that $C_1,...,C_k$ are $(-2)$-curves, $E_k^2=-4$, $C_k\cdot E_k=2$ and $E_k$ is disjoint from $(C_0+ \cdots +C_{k-1})$. 

The condition $C_k\cdot E_k=2$ implies that $C_k$ and $E_k$ intersect at one or two points. If they intersect at one point, then  we can decompose $E_k$ into $2C_{k+1}+E_{k+1}$ such that $C_{k+1}$ is a smooth rational curve, $C_k\cdot C_{k+1}=1$ and  $E_{k+1}\cdot C_k=0$. Then $$\tilde{p}^*b=2(C_0+C_1+\cdots +C_{k+1})+E_{k+1}.$$ From Lemma \ref{exactly-one--1-curve}.(2), we know that $C_{k+1}$ is  a $(-2)$-curve. We also  have $$C_{k+1}\cdot E_k =C_{k+1}\cdot (\tilde{p}^*b-2(C_0+\cdots C_k))=-2C_{k+1} \cdot C_k=-2.$$  Thus $$E_{k+1}^2=(E_k-2C_{k+1})^2=-4 \ \mathrm{and} \ C_{k+1} \cdot E_{k+1}= C_{k+1} \cdot (E_k-2C_{k+1})=2.$$ We are in the same situation as before. In this case, we repeat the same procedure.

Since there are only finitely many components in $\tilde{p}^*b$,   we can find  a positive integer $i$ such that $$\tilde{p}^*b=2(C_0+C_1+\cdots +C_i)+E_i$$ and $C_i$ intersects $E_i$ at two different points. As in the proof of Lemma \ref{double-fibre-canonical-2-sing}, we can write $E_i=D+D'+R$, where $D$, $D'$ are smooth rational curves such that $D\cdot D'=0$, $R\cdot C_i=0$,  $D\cdot C_i=1$, and $D'\cdot C_i=1$.  Hence we obtain the decomposition $(*)$

Now we will prove that $R=0$. There is a birational morphism $\widetilde{X}\to W$ which contracts exactly $C_1, ..., C_{i-1}$. The surface $W$ is smooth and there is a fibration $q:W\to B$ induced by $\tilde{p}$. We still denote by $C_i$, $D$, $D'$ and $R$ their strict transformations  in $W$. Then $C_i$  is the unique $(-1)$-curve in $q^*b$ and $q^*b=2C_i+D+D'+R$.  Similarly to the proof of Lemma \ref{double-fibre-canonical-2-sing}, we can prove that $D^2=D'^2=-2$, and $R=0$. Hence, the singularity of the singular point is canonical of type $D_{i+2}$.
\end{proof}

\subsection{Proof of Theorem \ref{thm-rat-surface} and Theorem \ref{main-thm}}
\label{proof-surface}

We will first  prove two theorems on classification of rational surfaces (Theorem \ref{thm-surface-1} and Theorem \ref{thm-surface-2}).

\begin{thm}
\label{thm-surface-1}
Let $X$ be a Fano surface with klt singularities. Assume that $h^0(X,\sO_X(-K_X))\geqslant 2$ and $\mathrm{Cl}(X)=\mathbb{Z}\cdot [K_X]$. Then either $X \cong S^c(E_8)$ or $X \cong S^n(E_8)$.
\end{thm}

\begin{proof}
If $X$ has canonical singularities, then $K_X$ is a Cartier divisor. Hence $\mathrm{Pic}(X)\cong \mathrm{Cl}(X)$. This  implies that $X$ has one singular point of type $E_8$ by \cite[Lem. 6]{MZ88}. Hence either $X \cong S^c(E_8)$ or $X \cong S^n(E_8)$.

Now we will  prove that $X$ has canonical singularities. First note that $X$ is singular.  Let $T$ be the curve given by a general member of $H^0(X,\sO_X(-K_X))$. Then $T$ is a minimal curve.

Let $Z$ be the surface described in Proposition \ref{aux-surface}. Then there are birational morphisms $\widetilde{X} \longrightarrow Z \overset{\pi}{\longrightarrow} X$  and a fibration $p:Z \to \p^1$, where $\widetilde{X} \to X$ is the minimal resolution of $X$.

\centerline{
\xymatrix{
Z  \ar[r]^{\pi}  \ar[d]^{p}
&X  
\\ 
\p^1 
& 
}
}

Let $F$ be a general fiber of the fibration $p:Z\to \p^1$. We denote the $\pi$-exceptional curves by $E_1,...,E_s$. Since $X$ has klt singularities, and the minimal resolution $\widetilde{X}\to X$ factors through $\widetilde{X}\to Z$, we obtain that $K_Z=\pi^*K_X+\sum_{i=1}^s a_iE_i$ with $-1<a_i\leqslant 0$ for all $i$. Hence, $$\pi^*T \cdot F=-\pi^*K_X\cdot F = -K_Z \cdot F + \sum_{i=1}^s a_iE_i\cdot F=2 + \sum_{i=1}^s a_iE_i\cdot F.$$ Since all of the $E_i$'s are horizontal over $\p^1$ with respect to the fibration $p$, we have $\sum a_iE_i\cdot F \leqslant 0$. Thus $\pi^*T \cdot F \leqslant 2$, and the equality holds if and only if all of the $a_i$'s are equal to $0$.

Moreover, we have $\pi^*T=\pi_*^{-1}T+\sum_{i=1}^s b_iE_i$ with $b_i \geqslant 0$ for all $i$. Hence $$\pi_*^{-1}T \cdot F= \pi^*T \cdot F - \sum_{i=1}^s b_iE_i\cdot F \leqslant \pi^*T \cdot F \leqslant 2$$ for the same reason as before.

First assume that $\pi_*^{-1}T \cdot F < 2$. Then $\pi_*^{-1}T$ is either a section or a fiber of $p$. Hence $T$ is a rational curve. Let $X \dashrightarrow \p^1$ be the rational map induced by a general $1$-dimensional linear subsystem of the linear system $|-K_X|$ which contains $T$. Let $Y$ be the normalisation of the graph of this rational map. Let $p_1:Y \to X$ and $p_2:Y \to \p^1$ be natural projections. Note that if a curve in $Y$ is contained in the fiber of $p_2$, then it is not contracted by $p_1$ since the graph of $X \dashrightarrow \p^1$ is contained in $X \times \p^1$, and the normalisation map is finite. Hence every fiber of $p_2:Y \to \p^1$ is reduced and irreducible for $-K_X$ is  minimal. Let $T'$ be the strict transform of $T$ in $Y$. Then $T'$ is smooth since $Y$ is normal and $T$ is a general member of $|-K_X|$. Since $T$ is a rational curve, we have $T'\cong \p^1$ and  $p_2$ is a $\p^1$-bundle by Lemma \ref{family-smooth-curves}.   In this case, $Y\cong \Sigma_k$ is a Hirzebruch surface. Since $X$ is singular, we obtain that $k\geqslant 2$ and that $p_1:Y \to X$ is the morphism which contracts the special rational curve in $Y$ with self intersection number $-k$. Thus $-K_X$ cannot be a minimal curve  by Lemma \ref{can-degree-formula} and we obtain a contradiction.

Hence we must have $\pi_*^{-1}T \cdot F = 2$. This shows that $a_i=b_i=0$ for all $i$. By the construction of $Z$, this implies that there is exactly one $\pi$-exceptional curve $D$ in $Z$. Since $X$ has Picard number $1$, $Z$ has Picard number $2$. Hence $p$ is a Mori fibration, and every fiber of $p$ is irreducible.

\centerline{
\xymatrix{
\widetilde{X}   \ar[d]_{\mathrm{minimal \ resolution}} \ar[dr]^{\mathrm{minimal \ resolution}}
& 
\\
Z  \ar[r]^{\pi}_{\mathrm{crepant}} \ar[d]^{p}_{\mathrm{Mori \ fibration}}
&X  
\\ 
\p^1 
& 
}
}

Moreover, since $b_i=0$ for all $i$, the curve $T$ is contained in the smooth locus of $X$. In particular, we have $$-K_X\cdot T\geqslant 1.$$

In order to prove that $X$ has canonical singularities, it is enough to prove that $Z$ has canonical singularities since $\pi^*K_{X}=K_Z$. Note that if every fiber of $p$ has multiplicity at most equal to $2$, then by Proposition \ref{double-fibre-canonical}, $Z$ has canonical singularities.  Now we will assume by contradiction that there is a point $b \in \p^1$ such that $p^*b = mC$ with $C$ reduced  and $m>2$. Since $a_i=0$ for all $i$, by the projection formula we have $$-K_X \cdot \pi_*C=-\pi^*K_X \cdot C =  -K_Z \cdot C = \frac{2}{m}< 1 \leqslant -K_X\cdot T.$$ 

This is a contradiction since $T$ is a minimal curve. Thus  $X$ has canonical singularities. This completes the proof of the theorem
\end{proof}

\begin{thm}
\label{thm-surface-2}
Let $X$ be a Fano surface with klt singularities and let $T$ be a minimal curve. Assume that $\mathrm{Cl}(X)\cong \mathbb{Z}\cdot [T]$, $h^0(X,\sO_X(T))\geqslant 3$ and   $h^0(X, \sO_X(T+K_X)) =0$. Then $X\cong \p^2$. 
\end{thm}

\begin{proof}
It is enough to prove that $X$ is smooth. Assume that $X$ is singular. Since $h^0(X,\sO_X(T)) \geqslant 3$, through a general point of $X$, there passes at least two different minimal curve. Hence $T^2 \geqslant 1$ by Lemma \ref{intersection-smooth-point}. Let $\pi:Z \to X$ be the surface described in Proposition \ref{aux-surface}. 

\centerline{
\xymatrix{
Z  \ar[r]^{\pi}  \ar[d]^{p}
&X  
\\ 
\p^1 
& 
}
}

Let $F$ be a general fiber of the fibration $p:Z\to \p^1$. We denote the $\pi$-exceptional curves by $E_1,...,E_s$. As in the  proof of Theorem \ref{thm-surface-1}, we have $$-K_X \cdot \pi_*F  =-\pi^*K_X\cdot F = -K_Z \cdot F + \sum_{i=1}^s a_iE_i\cdot F=2 + \sum_{i=1}^s a_iE_i\cdot F\leqslant 2.$$ 

If $\pi_*F$ is a minimal curve, then every fiber of $p$ is reduced and irreducible. Thus $Z\cong \Sigma_k$ is a Hirzebruch surface  Lemma \ref{family-smooth-curves}. Since $X$ is singular, we obtain that $k\geqslant 2$ and that $\pi:Z \to X$ is the morphism which contracts the special rational curve in $Z$ with self intersection number $-k$. Hence, by Lemma \ref{can-degree-formula}, we have $$1\leqslant T \cdot T=\pi_*F\cdot \pi_*F=\frac{1}{k}.$$  This is a contradiction.

If $\pi_*F$ is not a minimal curve in $X$, then there is an integer $\alpha>1$ such that $\pi_*F$ is linearly equivalent to $\alpha T$. Moreover, since  $-K_X$ is  $\mathbb{Q}$-ample, and $h^0(X, \sO_X(T+K_X)) =0$, there is an integer $\beta > 1$ such that $-K_X$ is linearly equivalent to $\beta T$. Hence, $-K_X \cdot \pi_*F = \alpha \beta T^2 \geqslant 4$, yielding a contradiction.
\end{proof}

After Theorem \ref{thm-surface-1} and Theorem \ref{thm-surface-2}, in order to prove Theorem \ref{thm-rat-surface}, it remains to prove that $S^c(E_8)$ does not satisfy the conditions of Theorem \ref{thm-rat-surface}. Let $X$ be a surface isomorphic to $S^c(E_8)$. Let $T$ be a minimal curve in $X$. Then $\sO_X(T)\cong \sO_X(-K_X)$. We will show that $h^0(X, \Omega_{X}^{[1]}[\otimes] \sO_{X}(T))\geqslant 1$. In order to do this, we will need the following results.

\begin{lemma}
\label{isotrivial-E8}
Let $X$ be a surface isomorphic to $S^c(E_8)$. Then all of the elliptic curves in the linear system $|-K_X|$  are isomorphic.
\end{lemma}

\begin{proof}
We can obtain $X$ after some birational transform from $\p^2$ (See section \ref{surface-pre}). There is a unique cuspidal rational curve in  the linear system $|-K_X|$ which is contained in the smooth locus of $X$.  Let $C$ the strict transform of this curve in $\p^2$. Then there is a unique smooth inflection point $x$ on $C$.    We may assume that $C$ is given by the equation $a_3a_2^2=a_1^3$, and that $x$ is the point $[0:1:0]$, where $[a_1:a_2:a_3]$ are coordinates of $\p^2$. Let $E$ be the strict transform in $\p^2$ of a smooth elliptic curve in the family induced by $|-K_X|$. Then the intersection of $E$ and $C$ is the point $x$ since $|-K_X|$ has a unique basepoint. Assume that $E$ is given by the homogeneous equation $$\lambda_1a_1^3+ \lambda_2a_2^3+ \lambda_3a_3^3+ \lambda_4a_1^2a_2+ \lambda_5a_2^2a_3+ \lambda_6a_3^2a_1+ \lambda_7a_2^2a_1+ \lambda_8a_3^2a_2+ \lambda_9a_1^2a_3+ \lambda_{10}a_1a_2a_3=0.$$ With affine coordinates $(b_1,b_3)=(\frac{a_3}{a_2},\frac{a_1}{a_2})$, this equation becomes $$\lambda_1b_1^3+ \lambda_2+ \lambda_3b_3^3+ \lambda_4b_1^2 + \lambda_5b_3+ \lambda_6b_3^2b_1+ \lambda_7b_1+ \lambda_8b_3^2+ \lambda_9b_1^2b_3+ \lambda_{10}b_1b_3=0.$$
The equation defining $C$ becomes $b_3=b_1^3$. If we replace $b_3$ by $b_1^3$ in the equation defining $E$, then we obtain $$\lambda_1b_1^3+ \lambda_2+ \lambda_3b_1^9+ \lambda_4b_1^2 + \lambda_5b_1^3+ \lambda_6b_1^7+ \lambda_7b_1+ \lambda_8b_1^6+ \lambda_9b_1^5+ \lambda_{10}b_1^4=0.$$ This equation  should have $b_1=0$ as a root with multiplicity $9$ since $E$ and $B$ intersects at $x$ with multiplicity $9$. Hence we have $\lambda_1+\lambda_5=0$, and $\lambda_k=0$ for $k\neq 1,3,5$. Since $E$ is smooth, $\lambda_1 \neq 0$. We may assume that $\lambda_1 =1$, then   $E$ is given by $a_3a_2^2=a_1^3 + \lambda a_3^3$ with $\lambda \neq 0$. Hence the $j$-invariant of $E$ is $0$. Thus all of the elliptic curves in this family induced by  $|-K_X|$ are isomorphic. 
\end{proof}

\begin{prop}
\label{picard-number-surface}
Let $X$ be a klt projective rational surface. Let $\rho$ be the Picard number of $X$. Then $h^1(X,\Omega_{X}^{[1]})=\rho$.
\end{prop}

\begin{proof}
First assume that $X$ is smooth. We have the following exponential exact sequence $$0\to \mathbb{Z} \overset{\cdot 2i\pi}{\longrightarrow} \sO_{X,\mathrm{an}} \overset{\mathrm{exp}}{\longrightarrow} \sO_{X,\mathrm{an}}^* \to 0.$$ By taking the cohomology and the GAGA principal (See \cite[Thm.1 in \S 12]{Ser56}), we obtain an exact sequence   $$H^1(X, \sO_X) \to H^1(X, \sO_X^*) \to H^2(X, \mathbb{Z}) \to H^2(X, \sO_X).$$ Since $X$ is a smooth rational surface, we have $h^1(X, \sO_X)=h^2(X, \sO_X)=0$. Hence $H^1(X, \sO_X^*) \to H^2(X,\mathbb{Z})$ is an isomorphism. Since  $H^1(X, \sO_X^*)$ is isomorphic to the Picard group of $X$, we obtain that $\rho$ is equal to $b_2(X)$, the second Betti number of $X$. Since $h^2(X,\sO_X)=0$, from the Hodge theory, we have $h^1(X,\Omega_X^1)=b_2(X)$. Hence $h^1(X,\Omega_X^1)=\rho$.

Now we assume that $X$ has klt singularities, and let $r:\widetilde{X} \to X$ be the minimal resolution of singularities. Let $\rho_{\widetilde{X}}$ be the Picard number of $\widetilde{X}$. Then from the previous paragraph, we know that $b_2(\widetilde{X})=\rho_{\widetilde{X}}$. Since $X$ is $\mathbb{Q}$-factorial, we have $\rho=\rho_{\widetilde{X}}-l$, where $l$ is the number of the $r$-exceptional curves. Hence $\rho=b_2(\widetilde{X})-l$.

Since $X$ has rational singularities, $R^1r_*\sO_{\widetilde{X}}=0$. Hence the Leray spectral sequence for $r$ gives an exact sequence $$0\to H^2(X,\mathbb{C})\to H^2(\widetilde{X},\mathbb{C})\overset{p}{\to}H^0(X,R^2r_*\mathbb{C}),$$ see \cite[12.1.3.2]{KM92}. Moreover, from \cite[Prop. 12.1.6]{KM92}, the image of $p$ is free and generated exactly by the class of the the $r$-exceptional divisors. Hence we obtain that $b_2(X)=b_2(\widetilde{X})-l$ and $\rho=b_2(X)$.

Since $X$ has quotient singularities (See \cite[Prop. 4.18]{KM98}), by \cite[Thm. 1.12]{Ste76}, we have $$b_2(X)=h^0(X,\Omega_X^{[2]})+h^1(X,\Omega_X^{[1]})+h^2(X,\sO_X)=h^1(X,\Omega_X^{[1]})+2h^2(X,\sO_X).$$ Since $X$ is a rational surface with rational singularities, from the Leray spectral sequence, we have $$h^2(X,\sO_X)=h^2(\widetilde{X},\sO_{\widetilde{X}})=0.$$ Thus, $\rho=b_2(X)=h^1(X,\Omega_X^{[1]})$.
\end{proof}

Now we will complete the proof of Theorem \ref{thm-rat-surface}.

\begin{proof}[{Proof of Theorem \ref{thm-rat-surface}}] Let $T$ be a minimal curve. From the discussion  after Lemma \ref{Weil-divisor-class}, we know that  one of the following properties holds
\begin{enumerate}

\item   $h^0(X,\sO_X(-K_X))\geqslant 2$ and $\mathrm{Cl}(X)=\mathbb{Z}\cdot [K_X]$; 

\item   $h^0(X,\sO_X(T))\geqslant 3$, $h^0(X, \sO_X(T+K_X))=0$, and $\mathrm{Cl}(X)=\mathbb{Z}\cdot [T]$

\end{enumerate}

By Theorem \ref{thm-surface-1} and Theorem \ref{thm-surface-2}, we only have to prove that $X$ is not isomorphic to $S^c(E_8)$. We will argue by contradiction. Assume that $X\cong S^c(E_8)$.

Let $E$ be a smooth elliptic curve in $|-K_X|$. Let $c:Y \to X$ be the blow-up of the   basepoint $x$ of $|-K_X|$. Then we have a fibration $p:Y \to \p^1$ induced by $|-K_X|$. The general fibers of $p$ are smooth elliptic curves. Since $X \cong S^c(E_8)$,   this family of elliptic curves is isotrivial  by Lemma \ref{isotrivial-E8}.  Let $D$ be the strict transform  of $E$ in $Y$. Since $D$ is a fiber, $\Omega_Y^1|_D$ is an extension of $\sO_D$ by $\sO_D$. Since the family is isotrivial, the Kodaira-Spencer map is zero. Hence $T_Y|_D\cong T_D \oplus p^* T_{\p^1}|_D$. By taking the dual, we have $\Omega_Y^1|_D \cong \sO_D \oplus \sO_D$. Note that $X$ is smooth along $E$ and $Y$ is smooth along $D$. From the natural morphism $c^*\Omega_X^1 \to \Omega_Y^1$ and the conormal exact sequences, we obtain the following  commutative diagram with exact rows

\centerline{
\xymatrix{
0 \ar[r] & c^*Q_E \ar[r]\ar[d]^{\theta_1} & (c^*\Omega_X^1)|_D  \ar[d]^{\theta_2} \ar[r]& c^*\Omega_E^1 \ar[d] \ar[r] & 0\\
0 \ar[r]  & Q_D  \ar[r] & \Omega_Y^1|_D \ar[r]  & \Omega_D^1 \ar[r] & 0
}}

The   morphisms  $\theta_2$ and $\theta_1$ are injective, and the third is an isomorphism. By the snake lemma, we have $\mathrm{Coker}\,\theta_1 \cong \mathrm{Coker}\, \theta_2$. Note that $Q_D$ is isomorphic to $\sO_D$, and $Q_E$ is isomorphic to $\sO_E(-x)$. Hence $h^0(D,\mathrm{Coker}\, \theta_1)=1$. We obtain that $h^0(D,\mathrm{Coker}\, \theta_2)=1$. Hence $$h^0(E, \Omega_X^1|_E ) = h^0(D, (c^*\Omega_X^1)|_D) \geqslant h^0(D, \Omega_Y^1|_D) - h^0(D,\mathrm{Coker}\, \theta_2) =1.$$ This shows that $\Omega_X^1|_E \cong \sO_E \oplus \sO_E(-x)$.

Moreover, since $\sO_X(K_X)|_E \cong \sO_E(-x)$, we obtain $(\Omega_X^1 [\otimes] \sO_X(-K_X))|_E \cong \sO_E \oplus \sO_E(x)$. Since $X$ is smooth along $E$ and $\sO_X(E)\cong \sO_X(-K_X)$, we have an exact sequence $$0 \to \Omega_X^{[1]} \to \Omega_X^{[1]} \otimes \sO_X(-K_X) \to (\Omega_X^1 [\otimes] \sO_X(-K_X))|_E \to 0.$$ This induces an exact sequence $$0\to H^0(X,\Omega_X^{[1]}) \to H^0(X, \Omega_X^{[1]} \otimes \sO_X(-K_X)) \to  H^0(E, (\Omega_X^1 [\otimes] \sO_X(-K_X))|_E) \to H^1(X,\Omega_X^{[1]}).$$
Since $X$ has Picard number $1$, by Proposition \ref{picard-number-surface}, we have $h^1(X,\Omega_X^{[1]})=1$. Moreover, since $h^0(X,\Omega_X^{[1]})=0$ (See \cite[Thm. 5.1]{GKKP11}), and $h^0(E, (\Omega_X^1 [\otimes] \sO_X(-K_X))|_E)= h^0(E,\sO_E \oplus \sO_E(x))=2$, we obtain $$h^0(X, \Omega_X^{[1]} \otimes \sO_X(-K_X))\geqslant 1.$$ Hence we have $h^0(X,\sO_X(-K_X)) -h^0(X, \Omega_X^{[1]} [\otimes] \sO_X(-K_X))  +h^0(X,\sO_X(-K_X+K_X)) \leqslant 2$, which is a contradiction.
\end{proof}

\begin{rem}
\label{rem-surface-E8}
We can prove that if $X\cong S^n(E_8)$, then $X$ satisfies the conditions in Theorem \ref{thm-rat-surface}. In fact, we can show that the family of elliptic curves induced by $|-K_X|$ is not isotrivial. Hence we have  $h^0(E, (\Omega_X^1 [\otimes] \sO_X(-K_X))|_E)=1$ in this case, where $E$ is a smooth elliptic curve induced by a general member of $|-K_X|$. Since $E$ is contained in the smooth locus of $X$, the composition of the natural morphisms $$H^2(X,\mathbb{C})\to H^2(\widetilde{X},\mathbb{C})\to H^2(E,\mathbb{C})$$ is not zero, where $\widetilde{X} \to X$ is the minimal resolution. Hence,   the natural morphism $H^1(X,\Omega_X^{[1]})\to H^1(E,\Omega_E^1)$ is not zero. Now consider the   exact sequence $$H^1(X,\Omega_X^{[1]}\otimes \sO_X(-E))\to H^1(X,\Omega_X^{[1]}) \to H^1(E, \Omega_X^{[1]}|_E).$$ Since the composition of $H^1(X,\Omega_X^{[1]}) \to H^1(E, \Omega_X^{[1]}|_E) \to H^1(E,\Omega_E^1)$ is not zero, and $h^1(X,\Omega_X^{[1]})=1$, we obtain that  $H^1(X,\Omega_X^{[1]}\otimes \sO_X(-E))\to H^1(X,\Omega_X^{[1]})$ in the exact sequence   is zero. By taking the dual morphism, we obtain that the natural morphism $H^1(X,\Omega_X^{[1]})\to H^1(X,\Omega_X^{[1]}\otimes \sO_X(-K_X))$ is zero. Hence, as in the proof of Theorem \ref{thm-rat-surface}, we obtain that $$h^0(X,\Omega_X^{[1]}\otimes \sO_X(-K_X))=0$$ in this case. Note that every ample divisor $T$ in $X$ is linearly equivalent to $-aK_X$ for some $a>0$. By induction on $a$, we can show that $X$ satisfies the conditions in Theorem \ref{thm-rat-surface}.
\end{rem}

Now we can deduce Theorem \ref{main-thm}.

\begin{proof}[{Proof of Theorem \ref{main-thm}}]
By Corollary \ref{cor-cohomology}, we have $$h^0(X,\sO_X(H)) -h^0(X, \Omega_X^{[1]} [\otimes] \sO_X(H))  +h^0(X,\sO_X(H+K_X)) = 3$$ for any $\mathbb{Q}$-ample Weil divisor $H$ on $X$. Since $M$ is simply connected in codimension $1$, the smooth locus of $X$ is algebraically simply connected. Hence, by Theorem \ref{thm-rat-surface}, either $X\cong \p^2$ or $X\cong S^n(E_8)$.
\end{proof}

\bibliographystyle{amsalpha}
\bibliography{references}

\end{document}